\documentclass[a4paper]{amsproc}
\usepackage{amssymb}
\usepackage{mathrsfs}
\usepackage{amsfonts}
\usepackage{amsthm}
\usepackage{amsmath,amscd}
\usepackage{amsxtra}     
\usepackage[all,cmtip]{xy}
\usepackage{epsfig}
\usepackage{verbatim}
\usepackage{color}
\usepackage{mathabx}
\usepackage{enumerate}
\usepackage{hyperref}
\usepackage{tikz}
\usetikzlibrary{arrows,matrix,shapes,trees}

\theoremstyle{plain}
 \newtheorem{thm}{Theorem}[section]
 \newtheorem{prop}{Proposition}[section]
 \newtheorem{lem}{Lemma}[section]
 \newtheorem{cor}{Corollary}[section]

 \newtheorem{lem'}{``Lemma''}[section]
 \newtheorem{fac}{Fact}[section]
\theoremstyle{definition}
 \newtheorem{ex}{Example}[section]
 \newtheorem{defn}{Definition}[section]
\theoremstyle{remark}
 \newtheorem{rmk}{Remark}[section]
 \numberwithin{equation}{section}

\renewcommand{\leq}{\leqslant}
\renewcommand{\geq}{\geqslant}

\newcommand{\N}{{\mathbb N}}
\newcommand{\Q}{{\mathbb Q}}
\newcommand{\Z}{{\mathbb Z}}

\newcommand{\Gm}{\mathbb{G}_m}

\newcommand{\Gml}{\mathbb{G}_{m,\mr{log}}}
\newcommand{\Tlog}{T_{\mathrm{log}}}

\newcommand{\mr}{\mathrm}
\newcommand{\mc}{\mathcal}

\title{Degenerating abelian varieties via log abelian varieties}


\author[H. Zhao]{\bfseries Heer Zhao}

\address{
Department of Pure Mathematics and Mathematical Statistics \\ 
University of Cambridge   \\ 
Cambridge\\
UK}
\email{heer.zhao@gmail.com}

\begin{document}

\vspace{18mm} \setcounter{page}{1} \thispagestyle{empty}

\begin{abstract}
For any split totally degenerate abelian variety over a complete discrete valuation field, we construct a log abelian variety, in the sense of \cite{k-k-n2}, over the discrete valuation ring extending the given abelian variety. This generalizes Kato's log Tate curve.
\end{abstract}

\maketitle

\section*{Introduction}
Degeneration appears naturally in compactifications of moduli spaces. Usually we prefer compactifications coming from moduli problems, in other words we prefer to use canonical (in a suitable sense) degenerate objects to make compactifications.

In the theory of classical toroidal compactifications of the moduli spaces of abelian varieties, there is no canonical choice of toroidal degenerations of abelian varieties. In late 80's, Kato in \cite[Sec. 2.2]{kat3} formulated a construction of log Tate curve, and conjectured the existence of a general theory of log abelian varieties. Later Kajiwara, Kato and Nakayama realised the theory of log abelian varieties in \cite{k-k-n1, k-k-n2}. Note that as indicated in \cite{k-k-n2}, there are other constructions of log abelian varieties in \cite{pah1, ols1}. However in this paper, we stick to the one defined in \cite{k-k-n2}. In some sense, a log abelian variety as a degeneration of a given abelian variety is to treat all possible toroidal degenerations of that abelian variety as ``one object'', hence it becomes canonical. This ``one object'' is proper, smooth, and even has a group structure on itself in the world of log geometry. These aspects make log abelian variety a perfect degeneration of abelian variety. For application of log abelian varieties, the short exact sequence in \cite[4.1.2]{k-k-n2} is the upshot.

Let $R$ be a complete discrete valuation ring with fraction field $K$, let $A_K$ be an abelian variety over $K$. Since log abelian varieties are supposed to be canonical degenerations of abelian varieties, there should be a canonical (unique) log abelian variety $A$ over $R$ extending $A_K$. As an example of log abelian variety, the authors of \cite{k-k-n2} constructed such a log abelian variety $\mc{E}_q$ over $R$ (or $O_K$ as in their notation) for the Tate curve $E_q$ over $K$ with ``$q$-invariant'' $q$, see \cite[1.6, 1.7, 4.7]{k-k-n2}. In this paper we generalise their log Tate curve to higher dimension case for split totally degenerate abelian varieties over complete discrete valuation fields. The difficulty of the generalisation lies in two aspects. Firstly, in the curve case, the formal toroidal models can  always be algebraized to schemes, whilst in the higher dimension case we have to turn to algebraic spaces which are more technical to handle. Secondly, as in most cases in mathematics, hard combinatorics shows up in higher dimension.

In the first section, we give the setting-up. In section 2, the main result is Theorem \ref{mainthm}, which says that the formal toroidal model $\mc{A}_{\Sigma}$ associated to any $Y$-admissible polytope decomposition $\Sigma$ algebraizes to an algebraic space $A_{\Sigma}$. Artin's theorems on ``existence of contractions and dilatations'' \cite{art2} are crucial for the proof. In section 3, we investigate the algebraic space $A_{\Sigma}$ in some details, and put a canonical log structure on it. The key point of this section is Corollary \ref{cor3.2}. In section 4, we construct a sheaf $A$, see (\ref{deflav}), and show that $A$ is the log abelian variety extending the given abelian variety $A_K$ over $K$ in Theorem \ref{thm4.2}. And the association of the log abelian variety $A$ to $A_K$ is actually a functor, see Theorem \ref{thm4.3}.

\section*{Acknowledgements }
I am grateful to my PhD supervisor Professor Anthony Scholl for introducing the beautiful work \cite{k-k-n2} to me, and also for encouragement and discussions during my PhD. I owe lots of thanks to Professor Martin Olsson, the discussion with whom brought Artin's work \cite{art2} into my attention. Part of this work was done when I was an informal guest at Professor B\"ockle's working group; I thank him for his kindness and hospitality. I would also thank Professor Jakob Stix and Professor Weizhe Zheng for helpful discussions.

\section{Setting-up}
Let $S$ be $\mr{Spec} R$, where $R$ is a complete DVR with fraction field $K$, a chosen uniformiser $\pi$ and residue field $k$. For each $n\in\N$, let $S_n$ be $\mr{Spec} R/({\pi})^{n+1}$ and $i_n$ the closed immersion $S_n\rightarrow S$. When $n$ is 0, we also use the notation $s$ (resp. $i$) for $S_0$ (resp. $i_0$). We regard $S$ as a log scheme with respect to the canonical log structure, i.e. the log structure associated to $\N\rightarrow R, 1\mapsto \pi$. We endow $S_n$ with the log structure induced from $S$. Let $j$ be the open immersion $\mr{Spec}K\rightarrow S$.

Let $(\mr{fs}/S)$ be the category of fs log algebraic spaces over $S$, and we regard it as a site endowed with the classical \'etale topology. Let $(\mr{fs}/S)'$ be the full subcategory of $(\mr{fs}/S)$ consisting of objects on which $\pi$ is locally nilpotent. We also endow $(\mr{fs}/S)'$ with the classical \'etale topology. For any fs log algebraic space $X$ over $S$, we don't distinguish the log algebraic space $X$ from the sheaf on $(\mr{fs}/S)$ represented by $X$.

Let $A_K$ be a semi-stable abelian variety over $K$ of dimension $d$ and let $A_K^*$ be the dual abelian variety of $A_K$, then we have the following two diagrams
\begin{equation}\label{diag1}
\xymatrix{
&&0\ar[d]  \\
&&Y\ar@{-->}[d]_{u_K}\ar[rd]^v  \\
0\ar[r] &T\ar[r] &\tilde{G}\ar[r]\ar@{-->}[d] &B\ar[r] &0 \\
&&A_K\ar[d]  \\
&&0
}
\end{equation}
and
\begin{equation}\label{diag2}
\xymatrix{
&&0\ar[d]  \\
&&X\ar@{-->}[d]_{u_K^*}\ar[rd]^{v^*}  \\
0\ar[r] &T^*\ar[r] &\tilde{G}^*\ar[r]\ar@{-->}[d] &B^*\ar[r] &0 \\
&&A_K^*\ar[d]  \\
&&0 . 
}
\end{equation}
The diagrams (\ref{diag1}) and (\ref{diag2}) are explained as follows:
\begin{enumerate}
\item[(a)] the rows in (\ref{diag1}) and (\ref{diag2}) are exact sequences of group schemes over $S$, which are the Raynaud extensions associated to $A_K$ and $A_K^*$ respectively. In particular, $T$ and $T^*$ are tori over $S$, and $B$ and $B^*$ are abelian schemes over $S$ dual to each other;
\item[(b)] the morphisms (labeled as dashed arrows) in the columns in
 (\ref{diag1}) and (\ref{diag2}) are defined rigid-analytically over
  $K$, but the morphisms $u_K$ and $u_K^*$ are also algebraic; $Y$
  (resp. $X$) is the character group of $T^*$ (resp. $T$), hence
  a locally constant sheaf over $S$ represented by a finite rank
  free $\Z$-module \'etale locally, and
  $\tilde{G}_K:=\tilde{G}\times_S K$ (resp. $\tilde{G}_K^*:=\tilde{G}^*\times_S K$) is the rigid analytic uniformization of $A_K$ (resp. $A_K^*$);
\item[(c)] $v$ (resp. $v^*$) is a morphism of group schemes over $S$ given by the $1$-motive dual of the Raynaud extension associated to $A_K^*$ (resp. $A_K$), and $u_K$ (resp. $u_K^*$) lifts $v$ (resp. $v^*$) over $K$.
\end{enumerate}

Via the duality theory of $1$-motives, the diagram (\ref{diag1}) (or equivalently (\ref{diag2})) is equivalent to another commutative diagram
\begin{equation}\label{diag3}
\xymatrix{
&&&Y\times_S X\ar@{-->}[ld]_{s_K}\ar[d]^{v\times v^*}  \\
0\ar[r] &\Gm\ar[r] &\mathscr{P}_B\ar[r] &B\times_S B^*\ar[r] &0 ,
}
\end{equation}
where the row in the diagram is the Poincar\'e biextension of
$(B,B^*)$ by $\Gm$ and $s_K$ is a bilinear section over $K$ along
$v\times v^*$.


From now on, we assume that $T$ is a split torus (in other words $X$ is constant), $Y$ is constant, and $A_K$ is totally degenerate, i.e. $B$ is zero (in the future we will deal with the general case). And in this case, we say that $A_K$ is \textbf{split totally degenerate}. Then the bilinear section $s_K$ is just a bilinear pairing 
\begin{equation}\label{eq1.1}
<,>:X\times Y\rightarrow K^{\times}.
\end{equation} 
Here we switch the positions of $Y$ and $X$ for coincidence with \cite{k-k-n2}). Composing the paring (\ref{eq1.1}) with the valuation map of $K$, we get a pairing 
\begin{equation}\label{eq1.2}
<,>:X\times Y\rightarrow \Z.
\end{equation} 
Since (\ref{eq1.1}) and (\ref{eq1.2}) are closely related, we denote both of them by $<,>$ by abuse of notation.

\section{Proper models associated to admissible polytope decompositions}
 As in \cite[Section 6]{mum1}, we study the (convex) polytope decompositions of the affine space $E:=\mr{Hom}(X,\Q)$. For the definition of polytopes and other notion concerning polytopes, we refer to \cite[Appendix]{oda1}.

\begin{defn}\label{def2.1}
A \textbf{polytope decomposition} $\Sigma$ of $E$ is a set of polytopes $\sigma\subset E$ such that
\begin{enumerate}
\item[(1)]$\cup_{\sigma\in\Sigma}\sigma=E$;
\item[(2)]if $\tau\leq\sigma$ and $\sigma\in\Sigma$, then $\tau\in\Sigma$;
\item[(3)]if $\sigma,\tau\in\Sigma$ with $\sigma\cap\tau\neq\emptyset $, then $\sigma\cap\tau$ is a common face of $\sigma$ and $\tau$.
\end{enumerate}

Given another polytope decomposition $\Sigma'$ of $E$, we say that there exists a \textbf{map from $\Sigma'$ to $\Sigma$} if for any $\sigma'\in\Sigma'$ there is a $\sigma\in\Sigma$ such that $\sigma'\subseteq\sigma$. It is easy to see that there exists at most one map from $\Sigma'$ to $\Sigma$. And if there exists one, it realises $\Sigma'$ as a subdivision of $\Sigma$.

Let $H$ be a group acting on $E$, a polytope decomposition $\Sigma$ is called \textbf{$H$-stable}, if $h\cdot\sigma\in\Sigma$ for any $\sigma\in\Sigma$ and any $h\in H$. If moreover $\Sigma$ has only finitely many orbits, then $\Sigma$ is called \textbf{$H$-admissible} (or simply \textbf{admissible} if the underlying group is clear in the context).

If $\Sigma'$ is an $H'$-stable polytope decomposition for another group $H'$ which acts on $E$ too, and we are given a group homomorphism $f:H'\rightarrow H$, then a map from $\Sigma'$ to $\Sigma$ is called \textbf{equivariant}, if it is compatible with the homomorphism $f$ and the group actions, i.e. $h'\cdot\sigma'\subset f(h')\cdot\sigma$ for any $h'\in H',\sigma'\in\Sigma',\sigma\in\Sigma$ such that $\sigma'\subset\sigma$ (we will be particularly interested in the case that $H'$ is a subgroup of $H$). A \textbf{map from the pair $(H',\Sigma')$ to the pair $(H,\Sigma)$} is defined to be an equivariant map from $\Sigma '$ to $\Sigma$. 
\end{defn}

The bilinear form (\ref{eq1.1}) $<,>: X\times Y\rightarrow K^{\times}$ realises $Y$ as a period lattice inside $T_K=\mc{H}om(X,\mathbb{G}_{m,K})$. Hence $Y$ acts on $T_K$ by translation as well as on $E=\mr{Hom}(X,\Q)$. We will often work with the $Y$-action on $E$, hence being admissible will always mean being $Y$-admissible if the acting group is not specified.

\begin{fac}\label{fac2.1}
Given any admissible polytope decomposition $\Sigma$ of $E$, by\\
\cite[Chap. IV, Sect. 3]{k-k-m-s} and \cite[Cor. 6.6]{mum1} we get a normal scheme $P_{\Sigma}$ locally of finite type over $S$ such that:
\begin{enumerate}
\item[(a)] $P_{\Sigma,K}=T_K$.
\item[(b)] The translation action of $T_K$ on itself extends to a $T$-action on $P_{\Sigma}$, and $P_{\Sigma}$ can be covered by some $T$-invariant affine open sets $P_{\sigma}$, which are in one to one correspondence with $\sigma\in\Sigma$. Here $P_{\sigma}=\mr{Spec}A_{\sigma}$, where $A_{\sigma}=R[C(\sigma)^{\vee}\cap \mathbb{X}], \mathbb{X}=\pi^{\Z}\oplus X$ and $C(\sigma)$ is the cone in $\mathbb{E}:=\mr{Hom}(\mathbb{X},\Q)=
\Q\oplus E$ above $\sigma\subseteq E$ with $E$ identified with the hyperplane $(1,E)$ in $\mathbb{E}$.
\item[(c)] $P_{\sigma}\cap P_{\tau}=P_{\sigma\cap\tau}$ (resp. $P_{\sigma}\cap P_{\tau}=T_K$) for any $\sigma,\tau\in\Sigma$ with $\sigma\cap\tau\neq\emptyset$ (resp. $\sigma\cap\tau=\emptyset$).
\item[(d)] The torus $T$ naturally embeds into $P_{\Sigma}$ if and only if $\{0\}\in\Sigma$.
\item[(e)] For all valuations $v$ on $\mr{Frac}(T)$ if $v\geq 0$ on $R$ and if [for any $x\in X,\exists$ $n\in\Z$ such that $n\cdot v(\pi)\geq v(x)\geq -nv(\pi)$] hold, then $v$ has a centre on $P_{\Sigma}$.
\item[(f)] The action of $Y$ on $\Sigma$ gives rise to an action of $Y$ on $P_{\Sigma}$, via
$$S_y:P_{\sigma}\rightarrow P_{y+\sigma} $$
$$S_y^*:C(y+\sigma)^{\vee}\cap \mathbb{X}\rightarrow C(\sigma)^{\vee}\cap \mathbb{X}, \pi^n x\longmapsto<x,y>\pi^nx$$
for $y\in Y$ and $\pi^nx\in C(y+\sigma)^{\vee}\cap \mathbb{X}$ with $x$ the $X$-part. And the action induces an action on $P_{\Sigma,n}:=P_{\Sigma}\times_S S_n$ for each $n\in\N$.
\end{enumerate}
\end{fac}

Let $H$ be a subgroup of finite index of $Y$, then it is also a period lattice inside $T_K$. A $Y$-admissible polytope decomposition $\Sigma$ of $E$ is automatically $H$-admissible. The scheme $P_{\Sigma}$ from Fact \ref{fac2.1} is independent of whether $\Sigma$ is regarded  as being $Y$-admissible or $H$-admissible. Let $\tilde{\Sigma}$ be an $H$-admissible polytope decomposition of $E$, then by Fact \ref{fac2.1} we get a scheme $P_{\tilde{\Sigma}}$ endowed with an $H$-action. If there exists a map $\Sigma\rightarrow\tilde{\Sigma}$, i.e. $\Sigma$ is a subdivision of $\tilde{\Sigma}$, then we have a natural morphism of $S$-schemes $P_{\Sigma}\rightarrow P_{\tilde{\Sigma}}$ which is compatible with the $H$-actions.

\begin{prop}\label{prop2.1}
Let $H$ be a subgroup of finite index of $Y$, $\Sigma$ an $H$-admissible polytope decomposition of $E$. The quotient of the $H$-action on $P_{\Sigma,n}$ exists in the category of schemes over $S_n$.
\end{prop}
\begin{proof}
This is obvious, since $P_{\sigma,n}\cap P_{\tau,n}=P_{\sigma\cap\tau,n}$ (resp. $P_{\sigma,n}\cap P_{\tau,n}=\emptyset$) for $\sigma,\tau\in\Sigma$ with $\sigma\cap\tau\neq\emptyset$ (resp. $\sigma\cap\tau=\emptyset$).
\end{proof}

Let $(H,\Sigma)$ be as in Proposition \ref{prop2.1}. We take the quotient scheme of $P_{\Sigma,n}=P_{\Sigma}\times_S S_n$ by the $H$-action, and denote it by $A_{H,\Sigma,n}$. Passing to the colimit of $\{A_{H,\Sigma,n}\}_{n\in\N}$, we get a formal scheme $\mc{A}_{H,\Sigma}$ over $\mc{S}:=\mr{Spf} R$. We will abbreviate $\mc{A}_{Y,\Sigma}$ as $\mc{A}_{\Sigma}$ sometime, if there is no other group action involved in the context.

Let $\tilde{H}$ be a subgroup of finite index of $H$, $\tilde{\Sigma}$ an $\tilde{H}$-admissible polytope decomposition of $E$ such that $\tilde{\Sigma}$ is a subdivision of $\Sigma$, so we have a map $(\tilde{H},\tilde{\Sigma})\rightarrow(H,\Sigma)$. The map $\tilde{\Sigma}\rightarrow\Sigma$ induces a morphism $P_{\tilde{\Sigma}}\rightarrow P_{\Sigma}$ of $S$-schemes which is compatible with the $\tilde{H}$-actions, hence a morphism $\mc{A}_{\tilde{H},\tilde{\Sigma}}\rightarrow\mc{A}_{\tilde{H},\Sigma}$ of formal schemes over $\mc{S}$. We also have a canonical quotient morphism $\mc{A}_{\tilde{H},\Sigma}\rightarrow\mc{A}_{H,\Sigma}$ of formal schemes over $\mc{S}$. Taking the composition of these two morphisms, we get a canonical morphism $\mc{A}_{\tilde{H},\tilde{\Sigma}}\rightarrow\mc{A}_{H,\Sigma}$. If $\tilde{H}=H=Y$, we sometime abbreviate the morphism as $\mc{A}_{\tilde{\Sigma}}\rightarrow\mc{A}_{\Sigma}$.

\begin{defn}\label{def2.2}
Let $H$ be a subgroup of finite index of $Y$, we denote by $\mr{PolDecom}_{Y,H}$ the category of $H$-admissible polytope decompositions of $E$ with morphisms given by subdivisions. We denote by $\mr{PolDecom}_{Y}$ the category of pairs $(H,\Sigma)$, where $H$ is a subgroup of finite index of $Y$ and $\Sigma$ lies in $\mr{PolDecom}_{Y,H}$, and morphisms are just maps between pairs $(\tilde{H},\tilde{\Sigma})$ and $(H,\Sigma)$. 
\end{defn}

The category $\mr{PolDecom}_{Y,H}$ sits in the category $\mr{PolDecom}_{Y}$ as a full subcategory via $\Sigma\mapsto (H,\Sigma)$. The association of the formal scheme $\mc{A}_{H,\Sigma}$ over $\mc{S}$ to an element $(H,\Sigma)$ of $\mr{PolDecom}_{Y}$ gives rise to a functor 
\begin{equation}\label{formular1}
\mc{F}:\mr{PolDecom}_{Y}\rightarrow \{\mc{S}\text{-formal schemes} \},
\end{equation}
where $\{\mc{S}\text{-formal schemes} \}$ denotes the category of formal schemes over $\mc{S}$. By abuse of notation, we still denote by $\mc{F}$ the restriction of $\mc{F}$ to $\mr{PolDecom}_{Y,H}$.

\begin{thm}[Mumford\cite{mum1}, Alexeev and Nakamura \cite{a-n1},\\ Alexeev\cite{ale1}]\label{mumthm}
There exists a $Y$-admissible polytope decomposition $\Xi$ such that
the formal scheme $\mc{A}_{\Xi}$ admits an ample line bundle, hence
it is algebraisable, i.e. it comes from the formal completion of a
unique algebraic scheme $A_{\Xi}$ over $S$ along its special fibre.

Moreover, $A_{\Xi}$ is a stable semiabelic scheme under the
semiabelian scheme $G$ over $S$, where $G$ comes from the semistable
reduction theorem. And we can choose $\Xi$ such that $\{0\}\in\Xi$,
hence $A_{\Xi}$ contains $G$ as an open subscheme canonically.
\end{thm}

From now on, we fix such a $Y$-admissible polytope decomposition $\Xi$. The main result in this section is the following theorem.

\begin{thm}[Algebraisation of formal models]\label{mainthm}
Let $\Gamma$ be a $Y$-admissible polytope decomposition of $E$, the formal scheme $\mc{A}_{\Gamma}$ over $\mc{S}$ comes from the formal completion of a unique algebraic space $A_{\Gamma}$ over $S$ along its special fibre. 

Moreover, $A_{\Gamma}$ is a proper model (in the category of $S$-algebraic spaces) of $A_K$, i.e. the structure morphism of $A_{\Gamma}$ over $S$ is proper and the generic fibre $(A_{\Gamma})_K:=A_{\Gamma}\times_{S}K$ coincides with $A_K$.
\end{thm}

\begin{cor}\label{cor2.1}
There exits a functor
\begin{equation}\label{formular2}
M:\mr{PolDecom}_{Y}\rightarrow \{\text{proper algebraic}\; S\text{-spaces} \}
\end{equation}
from the category $\mr{PolDecom}_{Y}$ to the category of proper algebraic spaces over $S$, such that the functor $\mc{F}$ in (\ref{formular1}) factors through $M$.
\end{cor}
\begin{proof}
Let $H$ be a subgroup of finite index of $Y$, then $H$ is also a period lattice inside $T_K(K)$. Applying Theorem \ref{mainthm} to $H$ instead of $Y$, it is clear that we have a canonical functor $$M_H:\mr{PolDecom}_{Y,H}\rightarrow \{\text{proper algebraic}\; S\text{-spaces} \}$$
such that $\mc{F}:\mr{PolDecom}_{Y,H}\rightarrow \{\mc{S}\text{-formal schemes} \}$ factors through $M_H$. When $H$ runs in the set of subgroups of finite index of $Y$, we get functors $M_H$ which are compatible with each other. Hence they assemble into a functor
$$M:\mr{PolDecom}_{Y}\rightarrow \{\text{proper algebraic}\; S\text{-spaces} \}$$
with required properties.
\end{proof}

For $(H,\Gamma)\in \mr{PolDecom}_{Y}$, we denote the algebraic $S$-space $M((H,\Gamma))$ as $A_{H,\Gamma}$. In the case $H=Y$, we sometime use the simple notation $A_{\Gamma}$ when no confusion may arise.

The main ingredients of the proof of Theorem \ref{mainthm} are Artin's ``existence of contractions'' theorem and ``existence of dilatations'' theorem, see \cite[3.1, 3.2]{art2}. We start with the following proposition.

\begin{prop}\label{prop2.2}
Given a map $\iota:\Sigma\rightarrow\Gamma$ between two $Y$-admissible polytope decompositions of $E$ (note that $\iota$ has to be a $Y$-admissible subdivision), the morphism $$\mc{F}(\iota):\mc{A}_{\Sigma}\rightarrow\mc{A}_{\Gamma}$$
of $\mc{S}$-formal schemes is a formal modification in the sense of \cite[1.7]{art2}.
\end{prop}
\begin{proof}
We need to show that $\mc{F}(\iota)$ is proper and verifies the three conditions in \cite[Definition (1.7)]{art2}.

It is enough to show that $\mc{F}({\iota})_0:P_{\Sigma,0}/Y\rightarrow
P_{\Gamma,0}/Y$ is proper, and we use the valuative criterion. Given
any commutative diagram
\begin{equation}\label{diag4}
\xymatrix{
\eta\ar[r]^-{\alpha}\ar@{^{(}->}[d]_j &P_{\Sigma,0}/Y\ar[d]^{\mc{F}({\iota})_0} \\
V\ar[r]^-{\beta}       &P_{\Gamma,0}/Y
}
\end{equation}
with $V$ the spectrum of a discrete valuation ring and $\eta$ the
open point of $V$. Then there exists a $\bar{\sigma}\in\Sigma/Y$
(resp. $\bar{\gamma}\in\Gamma/Y$) such that $\alpha(\eta)$ (resp.
$\beta(\eta)$) lies in the corresponding $T_0$-orbit $O_{\bar{\sigma}}$ ($O_{\bar{\gamma}}$) in $P_{\Sigma,0}/Y$ (resp.
$P_{\Gamma,0}/Y$). Then we have that
$$\overline{O_{\bar{\sigma}}}=\bigcup_{\{\bar{\sigma}_i\in\Sigma/Y|\bar{\sigma}\subset\bar{\sigma}_i\}} O_{\bar{\sigma}_i},
\quad
\overline{O_{\bar{\gamma}}}=\bigcup_{\{\bar{\gamma}_j\in\Gamma/Y|\bar{\gamma}\subset\bar{\gamma}_j\}}
O_{\bar{\gamma}_j}$$ and the morphism $\beta$ factors through
$\overline{\beta(\eta)}$. Choose suitable liftings $\sigma_i$'s
(resp. $\gamma_j$'s) of $\bar{\sigma}_i$'s (resp.
$\bar{\gamma}_j$'s) such that all $\sigma_i$'s (resp.$\gamma_j$'s)
contain the lifting $\sigma$ (resp. $\gamma$) and
$\cup_i\sigma_i\subset\cup_j\gamma_j\subset E$. Then the diagram
(\ref{diag4}) lifts to a commutative diagram
\begin{equation}\label{diag5}
\xymatrix{
\eta\ar[r]^{\tilde{\alpha}}\ar@{^{(}->}[d]_j &P_{\Sigma,0}\ar[d] \\
V\ar[r]^{\tilde{\beta}}\ar@{-->}[ru]^{\tilde{\delta}}       &P_{\Gamma,0}
}
\end{equation}
in which the morphism $P_{\Sigma,0}\rightarrow P_{\Gamma,0}$ is
proper, hence $\tilde{\alpha}$ factors through some
$\tilde{\delta}$. And $\tilde{\delta}$ factors through
$\bigcup_{\{\sigma_i\}\subset\Sigma}O_{\sigma_i}$, so gives rise to
a morphism $\delta:V\rightarrow P_{\Sigma,0}/Y$. It's easy to see
that $\alpha$ factors through $\delta$. On the other hand if there
exists another morphism $\delta'$ such that $\alpha=\delta'\circ j$,
then we can lift $\delta'$ to a morphism
$\tilde{\delta}':V\rightarrow P_{\Sigma,0}$ such that
$\tilde{\alpha}=\tilde{\delta}'\circ j$. The properness of
$P_{\Sigma,0}\rightarrow P_{\Gamma,0}$ implies
$\tilde{\delta}'=\tilde{\delta}$, hence $\delta=\delta'$. Then the
properness of $\mc{F}(\iota)_0$ follows.

We have a morphism $P_{\Sigma}\rightarrow P_{\Gamma}$ of $S$-schemes induced by $\iota$. Let $\mc{P}_{\Sigma}:=\varinjlim_{n}P_{\Sigma,n}$ and $\mc{P}_{\Gamma}:=\varinjlim_{n}P_{\Gamma,n}$, then we have a morphism $\mc{P}_{\Sigma}\rightarrow\mc{P}_{\Gamma}$ of $\mc{S}$-formal schemes, which we still denote by $\iota$. By \cite[Corollary (1.15)]{art2}, $\iota$ is the formal modification induced by $P_{\Sigma}\rightarrow P_{\Gamma}$. We have the following Cartesian diagram
\begin{equation*}
\xymatrix{
\mc{P}_{\Sigma}\ar[r]\ar[d]_{\iota} &\mc{A}_{\Sigma}\ar[d]^{\mc{F}(\iota)}\\
\mc{P}_{\Gamma}\ar[r]                     &\mc{A}_{\Gamma}
}
\end{equation*}
with the rows \'etale coverings. Since the notion of formal modification is local on the base for the (formal) \'etale topology (see \cite[sixth line of the proof of Proposition (1.13)]{art2}), $\mc{F}(\iota)$ is a formal modification.
\end{proof}

\begin{proof}[Proof of Theorem \ref{mainthm}:]
Let $\Xi$ be as in Theorem \ref{mumthm}. We define $\Xi\sqcap\Gamma$ to be, similar as in \cite[5.2.15]{k-k-n1}, the set of polytopes of the form $\xi\cap\gamma$ for $\xi\in\Xi,\gamma\in\Gamma$. It is clear that $\Xi\sqcap\Gamma\in\mr{PolDecom}_{Y,Y}$, and we
 have the following diagram
 \begin{equation}\label{diag6}
  \xymatrix{
  &&\Xi\sqcap\Gamma\ar[dl]\ar[dr]  \\
  &\Xi & & \Gamma
  }
 \end{equation}
 in the category $\mr{PolDecom}_{Y,Y}$. Hence we get the following diagram
 \begin{equation}\label{diag7}
  \xymatrix{
  &&\mc{A}_{\Xi\sqcap\Gamma}\ar[dl]\ar[dr]  \\
  &\mc{A}_{\Xi} & & \mc{A}_{\Gamma}
  }
 \end{equation}
in the category of $\mc{S}$-formal schemes with the arrows being
formal modifications by Proposition \ref{prop2.2}. We know that $\mc{A}_{\Xi}$ algebraizes to an $S$-scheme $A_{\Xi}$ by Theorem \ref{mumthm}. By \cite[3.2]{art2} and \cite[3.1]{art2}, we get the following diagram
\begin{equation}\label{diag8}
  \xymatrix{
  &&A_{\Xi\sqcap\Gamma}\ar[dl]\ar[dr]  \\
  &A_{\Xi} & & A_{\Gamma}
  }
\end{equation}
of algebraic $S$-spaces, where $A_{\Xi\sqcap\Gamma}$ and $A_{\Gamma}$ are the algebraizations of $\mc{A}_{\Xi\sqcap\Gamma}$ and $\mc{A}_{\Gamma}$ respectively, and the two arrows are the corresponding modifications associated to the formal modifications in diagram (\ref{diag7}). The properness of $A_{\Gamma}$ over $S$ follows from the properness of $A_{\Xi}$ and the properness of the modifications. Since $(A_{\Xi})_K:=A_{\Xi}\times_{S}K$ is isomorphic to $A_K$, so is $(A_{\Gamma})_K$.
\end{proof}

Given a $Y$-admissible polytope decomposition $\Gamma$, and a subset $\bar{\Gamma}_0\subset\bar{\Gamma}:=\Gamma/Y$ which is stable under taking faces (i.e. $\bar{\gamma}\in\bar{\Gamma}$ and $\bar{\tau}\leq\bar{\gamma}$ imply that $\bar{\tau}\in\bar{\Gamma}$), we can associate an open algebraic subspace of $A_{\Gamma}$ to $\bar{\Gamma}_0$ as follows. Note that giving $\bar{\Gamma}_0$ is the same as giving a subset $\Gamma_0\subset\Gamma$ which is $Y$-stable and stable under taking faces. Then we have a closed subset
$$\tilde{B}=\bigcup_{\gamma\in\Gamma\backslash\Gamma_0} O_{\gamma}\subset P_{\Gamma}$$
and regard it as a reduced closed subscheme of $P_{\Gamma}$. Here $O_{\gamma}$ denotes the $T_0$-orbit corresponding to $\gamma$. The closed formal subscheme $\varinjlim_{n}(\tilde{B}\times_S S_n)/Y$ of $\varinjlim_{n}P_{\Gamma,n}/Y$ gives rise to a closed algebraic subspace $B$ of $A_{\Gamma}$. 

\begin{defn}\label{defn2.3}
The open algebraic subspace $A_{\Gamma,\bar{\Gamma}_0}$ of $A_{\Gamma}$ is defined to be the complement of the closed subspace $B$ of $A_{\Gamma}$. In the case that $\bar{\Gamma}_0=\{\bar{\tau}\,|\,\tau\leq\gamma\}$ for some $\gamma\in\Gamma$, we also use the notation $A_{\Gamma,\bar{\gamma}}$ for $A_{\Gamma,\bar{\Gamma}_0}$.
\end{defn}

\begin{rmk}\label{rmk2.1}
It is easy to see that $(A_{\Gamma,\bar{\gamma}})_{\bar{\gamma}\in\Gamma/Y}$ gives rise to a Zariski open covering of $A_{\Gamma}$.
\end{rmk}

\begin{prop}\label{prop2.3}
Let $\iota:\Sigma\rightarrow\Gamma$ be a map of $Y$-admissible polytope decompositions. Then for any subset $\bar{\Sigma}_0$ of $\bar{\Sigma}\cap\bar{\Gamma}$, which is stable under taking faces, $M(\iota)$ restricts to an isomorphism on $A_{\Sigma,\bar{\Sigma}_0}$.
\end{prop}
\begin{proof}
It is enough to consider the case $\bar{\Sigma}_0=\bar{\Sigma}\cap\bar{\Gamma}$. Since the restriction of $M(\iota)$ to $A_{\Sigma,\bar{\Sigma}_0}$ has set-theoretic image $A_{\Gamma,\bar{\Sigma}_0}$, we are further reduced to show that $M(\iota)|_{A_{\Sigma,\bar{\Sigma}_0}}$ is an open immersion. We make use of \cite[\href{http://stacks.math.columbia.edu/tag/0879}{Tag 0879}]{stacks-project} to achieve this. Firstly, the open subspace $A_K=A_{\Gamma}\times_SK$ of $A_{\Gamma}$ is clearly scheme theoretically dense, and the open immersion $A_K\rightarrow A_{\Gamma}$ is quasi-compact. Secondly, $M(\iota)_K:=M(\iota)\times_SK$ is the identity map of $A_K$. Thirdly, both $A_{\Sigma}$ and $A_{\Gamma}$ are proper over $S$, hence $M(\iota)$ is proper and in particular $M(\iota)|_{A_{\Sigma,\bar{\Sigma}_0}}$ is separated and of finite type; $A_{\Sigma}$ is flat over $S$, $M(\iota)|_{A_{\Sigma,\bar{\Sigma}_0}}$ is also flat by fibrewise criterion of flatness. Apply \cite[\href{http://stacks.math.columbia.edu/tag/0879}{Tag 0879}]{stacks-project} to $M(\iota)|_{A_{\Sigma,\bar{\Sigma}_0}}:A_{\Sigma,\bar{\Sigma}_0}\rightarrow A_{\Gamma}$, we have that $M(\iota)|_{A_{\Sigma,\bar{\Sigma}_0}}$ is an open immersion. 
\end{proof}

\begin{rmk}\label{rmk2.2}
Let $\gamma$ be a polytope of $E$ such that there exists a $Y$-admissible polytope decomposition $\Sigma$ containing $\gamma$. Then the algebraic space $A_{\Sigma,\bar{\gamma}}$ is independent of the choice of $\Sigma$ by Proposition \ref{prop2.3}. It is reasonable to denote $A_{\Sigma,\bar{\gamma}}$ by $A_{Y,\bar{\gamma}}$. The subindex $Y$ of $A_{Y,\bar{\gamma}}$ is aimed to indicate that: $\gamma$ is contained in some $Y$-admissible polytope decomposition of $E$; $A_{Y,\bar{\gamma}}$ has generic fibre $A_K$ which is the abelian varieties corresponding to the period lattice $Y$ inside $T_K$.
\end{rmk}

\begin{prop}\label{prop2.4}
Let $\Gamma$ be as in Theorem \ref{mainthm} with the additional condition $\{0\}\in\Gamma$. Then $A_{\Gamma}$ contains $G$ as an open algebraic subspace canonically, where $G$ is as in Theorem \ref{mumthm}.
\end{prop}
\begin{proof}
Let $\Xi$ be as in Theorem \ref{mumthm} such that $\{0\}\in\Xi$, then $A_{\Xi}$ contains $G$ as an open algebraic subspace canonically by Theorem \ref{mumthm}.
Let $\iota_1$ (resp. $\iota_2$) denote the subdivision $\Xi\sqcap\Gamma\rightarrow\Xi$ (resp. $\Xi\sqcap\Gamma\rightarrow\Gamma$) as in diagram (\ref{diag6}). Under our assumption we have $\{0\}\in\Xi\cap(\Xi\sqcap\Gamma)\cap\Gamma$, hence all the $Y$-translates $\{y\}$ of $\{0\}$ lie in $\Xi\cap(\Xi\sqcap\Gamma)\cap\Gamma$. Let 
$$Q_{\Xi}:=P_{\Xi}-\bigcup_{y\in Y}P_{\{y\}},\quad\mc{Z}_{\Xi}:=\varinjlim_{n}(Q_{\Xi}\times_{S}S_n)/Y,$$ 
$$Q_{\Xi\sqcap\Gamma}:=P_{\Xi\sqcap\Gamma}-\bigcup_{y\in Y}P_{\{y\}},\quad\mc{Z}_{\Xi\sqcap\Gamma}:=\varinjlim_{n}(Q_{\Xi\sqcap\Gamma}\times_{S}S_n)/Y,$$
$$Q_{\Gamma}:=P_{\Gamma}-\bigcup_{y\in Y}P_{\{y\}},\quad\mc{Z}_{\Gamma}:=\varinjlim_{n}(Q_{\Gamma}\times_{S}S_n)/Y,$$
here we regard $Q_{\Xi}$ (resp. $Q_{\Xi\sqcap\Gamma}$, resp. $Q_{\Gamma}$) as the reduced closed subscheme of $P_{\Xi}$ (resp. $P_{\Xi\sqcap\Gamma}$, resp. $P_{\Gamma}$) with underlying set $Q_{\Xi}$ (resp. $Q_{\Xi\sqcap\Gamma}$, resp. $Q_{\Gamma}$), and $\mc{Z}_{\Xi}$ (resp. $\mc{Z}_{\Xi\sqcap\Gamma}$, resp. $\mc{Z}_{\Gamma}$) is the closed formal subscheme of $\mc{A}_{\Xi}$ (resp. $\mc{A}_{\Xi\sqcap\Gamma}$, resp. $\mc{A}_{\Gamma}$) corresponding to $Q_{\Xi}$ (resp. $Q_{\Xi\sqcap\Gamma}$, resp. $Q_{\Gamma}$). By Grothendieck existence theorem \cite[Chap. 5, Sec. 6]{knu1}, $\mc{Z}_{\Xi}$ (resp. $\mc{Z}_{\Xi\sqcap\Gamma}$, resp. $\mc{Z}_{\Gamma}$) is the formal completion of a reduced closed algebraic subspace $Z_{\Xi}$ (resp. $Z_{\Xi\sqcap\Gamma}$, resp. $Z_{\Gamma}$) of $A_{\Xi}$ (resp. $A_{\Xi\sqcap\Gamma}$, resp. $A_{\Gamma}$). The algebraic subspace $Z_{\Xi}$ (resp. $Z_{\Xi\sqcap\Gamma}$, resp. $Z_{\Gamma}$) is supported on the special fibre, hence has the same support as $\mc{Z}_{\Xi}$ (resp. $\mc{Z}_{\Xi\sqcap\Gamma}$, resp. $\mc{Z}_{\Gamma}$).

Now both $M(\iota_1)$ and $M(\iota_2)$ restrict to an isomorphism over $A_{\Xi\sqcap\Gamma}-Z_{\Xi\sqcap\Gamma}$ by Proposition \ref{prop2.3}, hence $A_{\Gamma}-Z_{\Gamma}=A_{\Xi\sqcap\Gamma}-Z_{\Xi\sqcap\Gamma}=A_{\Xi}-Z_{\Xi}=G$.
\end{proof}

Now we describe some examples of models associated to certain polytope decompositions which are going to be used in later sections.
\begin{ex}\label{ex1}
Choose suitable basis for $X$ and $Y$ such that the pairing (\ref{eq1.2}) is given by a diagonal matrix $\mr{diag}(n_1,\cdots,n_d)$ \footnote{Note that this does not imply that the corresponding matrix for the pairing (\ref{eq1.1}) is diagonal. Actually if we can make the matrix diagonal under some choice of basis, then $A_K$ is isomorphic to a product of some Tate curves.}. Then $Y$ sits in $E$ as the lattice $$\mathop{\oplus}_i n_i\Z e_i$$ via the embedding $Y\hookrightarrow E$, where $e_1,\cdots,e_d$ is the basis of $E$ corresponding to the chosen basis of $X$. Consider the admissible polytope decomposition $\Sigma_{\Box^d}$ given by all the $Y$-translates of the faces of $\Box^d$, where $\Box^d$ is the $d$-cube with vertices $a_1e_1+\cdots+a_de_d$, $a_i\in\{0,n_i\}$. We get an algebraic space $A_{\Sigma_{\Box^d}}$ associated to $\Sigma_{\Box^d}$, and denote it also by $A_Y$.

Now we construct a model for $A_K\times A_K$. Consider the lattice $$Y\times Y=(\mathop{\oplus}_in_i\Z e_i)\oplus(\mathop{\oplus}_in_i\Z e_i)$$ in $E\times E$, let $\Box^{2d}$ be the $2d$-cube with vertices $$(a_1e_1+\cdots+a_de_d,b_1e_1+\cdots+b_d e_d)$$
with $a_i,b_i\in\{0,n_i\}$ for all $i$. Let $X_1,\cdots,X_d,Y_1,\cdots,Y_d$ be the sequence of coordinates of $E\times E$ with respect to the basis $e_1,\cdots,e_d$ of $E$. The hyperplanes $X_i-Y_i=0$ for $i=1,\cdots,d$, divide $\Box^{2d}$ into $2^d$ polytopes
$$\boxslash_{\underline{u}}=\Box^{2d}\cap(\bigcap_i H_{u_i}),$$ where $\underline{u}=(u_1,\cdots,u_d)\in\{0,1\}^d$ and $H_{u_i}$ denotes the half-space $X_i- Y_i\geq 0$ (resp. $X_i- Y_i\leq 0$) if $u_i=0$ (resp. $u_i=1$). Taking the $Y\times Y$-translates of the faces of the
$\boxslash_{\underline{u}}$'s, we denote the resulted $Y\times Y$-admissible polytope decomposition by $\Sigma_{\boxslash^{2d}}$ and the associated model of
$A_K\times A_K$ by $(A_Y\times_S A_Y)_{\boxslash^{2d}}$. Apparently we have a morphism of algebraic spaces
$$(A_Y\times_S A_Y)_{\boxslash^{2d}}\rightarrow A_Y\times_S A_Y$$
over $S$ coming from the above subdivision of $Y\times Y$-admissible polytope decompositions. Under the linear map $E\times E\rightarrow E, (a,b)\mapsto -a+b$, any $Y\times Y$-translate of a $\boxslash_{\underline{u}}$ is mapped into some $Y$-translate of $\Box^d$. It follows that we have a morphism
$$m_{-}:(A_Y\times_S A_Y)_{\boxslash^{2d}}\rightarrow A_Y$$
of algebraic spaces over $S$.

Similarly, we divide the polytope $\Box^{2d}$ into $2^d$ polytopes $\boxbackslash_{\underline{u}}$ by the hyperplanes $X_i+Y_i=n_i$ for $i=1,\cdots,d$, hence get a $Y\times Y$-admissible polytope decomposition $\Sigma_{\boxbackslash^{2d}}$ and its associated proper model $(A_Y\times_S A_Y)_{\boxbackslash^{2d}}$ together with morphisms 
$$(A_Y\times_S A_Y)_{\boxbackslash^{2d}}\rightarrow A_Y\times_S A_Y$$
and 
$$m_{+}:(A_Y\times_S A_Y)_{\boxbackslash^{2d}}\rightarrow A_Y,$$
where the corresponding map $E\times E\rightarrow E$ for the second morphism is given by $(a,b)\mapsto a+b$ for $a,b\in E$. 
\end{ex}

\begin{prop}\label{prop2.5}
The natural open immersion $G\times_S A_Y\hookrightarrow A_Y\times_S A_Y$ factors canonically as
\begin{equation*}
 \xymatrix{G\times_S A_Y\ar@{^{(}->}[rr]\ar@{^{(}->}[rd] & &A_Y\times_S A_Y \\
                      &(A_Y\times_S A_Y)_{\boxslash^{2d}}\ar[ru]
 }
\end{equation*}
and
\begin{equation*}
 \xymatrix{G\times_S A_Y\ar@{^{(}->}[rr]\ar@{^{(}->}[rd] & &A_Y\times_S A_Y \\
                      &(A_Y\times_S A_Y)_{\boxbackslash^{2d}}\ar[ru]
 }.
\end{equation*}
\end{prop}
\begin{proof}
The proof is similar to the proof of Proposition \ref{prop2.4} by considering all the $Y\times Y$-translates of $\{0\}\times\tau$'s for all faces $\tau$ of $\Box^d$, instead of considering all the $Y$-translates of $\{0\}$.
\end{proof}

We define a morphism $\rho$ as the composition of
$$G\times_S A_Y\hookrightarrow(A_Y\times_S A_Y)_{\boxbackslash^{2d}}\xrightarrow{m_{+}} A_Y,$$
then the morphism $\rho$ fits into the following commutative diagram
\begin{equation}
\xymatrix{G\times_S G\ar[rr]^{m_G}\ar@{^{(}->}[d] &&G\ar@{^{(}->}[d] \\
G\times_S A_Y\ar[rr]^{\rho}\ar@{^{(}->}[rd] && A_Y  \\
&(A_Y\times_S A_Y)_{\boxbackslash^{2d}}\ar[ru]_{m_{+}}
},
\end{equation}
where $m_G$ denotes the group law on $G$. The diagram suggests that we may expect $\rho$ to be a group action. This is indeed the case, and we will prove this after example \ref{ex2}. 

\begin{ex}\label{ex2}
Let the notation be as in Example \ref{ex1}. Now we construct some models for $A_K\times A_K\times A_K$. Consider the lattice 
$$Y\times Y\times Y=(\mathop{\oplus}_in_i\Z e_i)\oplus(\mathop{\oplus}_in_i\Z e_i)\oplus(\mathop{\oplus}_in_i\Z e_i)$$ 
in $E\times E\times E$, let $\Box^{3d}$ be the $3d$-cube with vertices
$$(a_1e_1+\cdots+a_de_d,b_1e_1+\cdots+b_d e_d,c_1e_1+\cdots+c_d e_d)$$
with $a_i,b_i,c_i\in\{0,n_i\}$. The $Y\times Y\times Y$-translates of the faces of $\Box^{3d}$ give rise to a $Y\times Y\times Y$-admissible polytope decomposition of $E\times E\times E$, and we denote it by $\Sigma_{\Box^{3d}}$. The model associated to $\Sigma_{\Box^{3d}}$ is just $A_Y\times_S A_Y\times_S A_Y$.

Let $X_1,\cdots,X_d,Y_1,\cdots,Y_d,Z_1,\cdots,Z_d$ be the sequence of coordinates of $E\times E\times E$ with respect to the basis $e_1,\cdots,e_d$ of $E$. By cutting $\Box^{3d}$ with the hyperplanes 
$$X_i+Y_i+Z_i=n_i,X_i+Y_i+Z_i=2n_i,X_i+Y_i=n_i,Y_i+Z_i=n_i$$
with $i$ varying from 1 to $d$, we get a subdivision of $\Box^{3d}$. Taking the $Y\times Y\times Y$-translates of this subdivision, we get a $Y\times Y\times Y$-admissible polytope decomposition of $E\times E\times E$ and denote it by $\Sigma_{\boxasterisk^{3d}}$. We denote the model associated to $\Sigma_{\boxasterisk^{3d}}$ by $(A_Y\times_S A_Y\times_S A_Y)_{\boxasterisk^{3d}}$. The decomposition $\Sigma_{\boxasterisk^{3d}}$ is clearly a subdivision of $\Sigma_{\Box^{3d}}$, whence a canonical morphism $(A_Y\times_S A_Y\times_S A_Y)_{\boxasterisk^{3d}}\rightarrow A_Y\times_S A_Y\times_S A_Y$.

The polytope decompositions $\Sigma_{\boxasterisk^{3d}},\Sigma_{\boxbackslash^{2d}}$ and $\Sigma_{\Box^d}$ are compatible with the commutativity of the diagram 
$$\begin{CD}
E\times E\times E @>{+_E\times 1_E}>> E\times E\\
@V{1_E\times +_E}VV @VV{+_E}V\\
E\times E @>{+_E}>> E \quad,
\end{CD}$$
where $+_E$ denotes the addition of $E$. Hence we get a commutative diagram
$$\begin{CD}
(A_Y\times_S A_Y\times_S A_Y)_{\boxasterisk^{3d}} @>{m_{+,12}}>> (A_Y\times_S A_Y)_{\boxbackslash^{2d}}   \\
@V{m_{+,23}}VV @VV{m_{+}}V\\
(A_Y\times_S A_Y)_{\boxbackslash^{2d}} @>{m_{+}}>> A_Y \quad .
\end{CD}$$
Similar as in Proposition \ref{prop2.5}, we have a canonical factorization
\begin{equation*}
 \xymatrix{G\times_S G\times_S A_Y\ar@{^{(}->}[rr]\ar@{^{(}->}[rd] & &A_Y\times_S A_Y\times_S A_Y \\
                      &(A_Y\times_S A_Y\times_S A_Y)_{\boxasterisk^{3d}}\ar[ru]
 }
\end{equation*}
with the two hooked arrows open immersions. Furthermore, the open immersions $G\times_S G\times_S A_Y\hookrightarrow (A_Y\times_S A_Y\times_S A_Y)_{\boxasterisk^{3d}}$ and $G\times_S A_Y\hookrightarrow (A_Y\times_S A_Y)_{\boxbackslash^{2d}}$ are compatible with the ``partial addition'' morphisms $m_{+,12}$ and $m_{+,23}$, in the sense that they fit into the following commutative diagram
\begin{equation}\label{diag9}
 \xymatrix{
 G\times_S G\times_S A_Y\ar[rr]^{m_G\times 1_{A_Y}}\ar@{^{(}->}[rd]\ar[dd]_{1_G\times\rho}&&G\times_S A_Y\ar@{^{(}->}[rd] &\\
   &(A_Y\times_S A_Y\times_S A_Y)_{\boxasterisk^{3d}}\ar[rr]^{m_{+,12}}\ar[dd]_{m_{+,23}} &&(A_Y\times_S A_Y)_{\boxbackslash^{2d}}\ar[dd]^{m_{+}}  \\
   G\times_S A_Y\ar@{^{(}->}[rd] &&&  \\
   &(A_Y\times_S A_Y)_{\boxbackslash^{2d}}\ar[rr]^{m_{+}} && A_Y
 }.
\end{equation}
\end{ex}

\begin{prop}\label{prop2.6}
The morphism $\rho:G\times_S A_Y\rightarrow A_Y$ defines a $G$-action on $A_Y$.
\end{prop}
\begin{proof}
The compatibility axiom for group action follows from the commutativity of the diagram (\ref{diag9}). We are left to check the role of the identity section $e_G:S\rightarrow G$, i.e. to check the commutativity of the following diagram
\begin{equation*}
\xymatrix{G\times_S A_Y \ar[r]^{\rho} &A_Y   \\
A_Y\ar[u]^{e_G\times 1_{A_Y}}\ar@{=}[ru]
}.
\end{equation*}
But $\rho\circ(e\times 1_{A_Y})=1_{A_Y}$ formally and $A_Y$ is proper over $S$, hence $\rho\circ (e\times 1_{A_Y})=1_{A_Y}$ by \cite[\href{http://stacks.math.columbia.edu/tag/0A4Z}{Tag 0A4Z}]{stacks-project}. 
\end{proof}

\begin{rmk}\label{rmk2.3}
For the model $A_{\Sigma}$ associated to a $Y$-admissible polytope decomposition $\Sigma$ with $\{0\}\in\Sigma$, we could ask if the translation action of $G$ on itself extends to $A_{\Sigma}$. The proposition \ref{prop2.6} offers an affirmative answer for the special case $A_Y$. For 1-dimensional case, this is known, see \cite{d-r1}. For the case that the decomposition $\Sigma$ provides an ample line bundle to $A_{\Sigma}$, the answer is mentioned to be yes in \cite[5.7.1]{ale1}, but the author could not find the arguments for proving this there. 
\end{rmk}

\section{The canonical logarithmic structure on toroidal models}
First of all, we investigate the algebraic space $A_{\Sigma}$ in more detail. 
\begin{lem}\label{lem3.1}
Let $A_{\sigma}$ and $P_{\sigma}=\mr{Spec}A_{\sigma}$ be as in Fact \ref{fac2.1} (b), let $\hat{A}_{\sigma}$ be the $\pi$-adic completion of $A_{\sigma}$, $\tilde{P}_{\sigma}=\mr{Spec}\hat{A}_{\sigma}$, and $\mc{P}_{\sigma}=\mr{Spf}\hat{A}_{\sigma}$. Then we have:
\begin{enumerate}[(i)]
\item The formal scheme $\mc{P}_{\sigma}$ is normal and Cohen-Macaulay, and the ring $\hat{A}_{\sigma}$ is normal and Cohen-Macaulay;
\item If $P_{\sigma}$ is regular, so are $\tilde{P}_{\sigma}$ and $\mc{P}_{\sigma}$.
\end{enumerate}
\end{lem}
\begin{proof}
For any point $x\in \mc{P}_{\sigma}$, we have a sequence of local homomorphisms of noetherian local rings
$$\mc{O}_{P_{\sigma},x}\xrightarrow{\delta}\mc{O}_{\tilde{P}_{\sigma},x}\xrightarrow{\lambda}\mc{O}_{\mc{P}_{\sigma},x}\xrightarrow{\mu}\hat{\mc{O}}_{\tilde{P}_{\sigma},x}=\hat{\mc{O}}_{P_{\sigma},x} \,,$$
where the completion $\hat{\cdot}$ means the $\pi$-adic completion. Both $\lambda$ and $\mu$ are faithful flat, see \cite[3.1.2]{g-m}. The composition $\mu\circ\lambda\circ\delta$ is just the canonical homomorphism from $\mc{O}_{P_{\sigma},x}$ to its completion.

As a complete discrete valuation ring, $R$ is an excellent ring, see \cite[7.8.3 (iii)]{egaIV-2}. Since $A_{\sigma}$ is a finitely generated $R$-algebra, it is excellent, so is $\mc{O}_{P_{\sigma},x}$, see \cite[7.8.3 (ii)]{egaIV-2}. Then by \cite[7.8.3 (v)]{egaIV-2}, we have $\hat{\mc{O}}_{P_{\sigma},x}$ is normal and Cohen-Macaulay, and it is also regular if $A_{\sigma}$ is regular at $x$. Since both $\mu$ and $\lambda$ are faithful flat, we have that $\mc{O}_{\mc{P}_{\sigma},x}$ and $\mc{O}_{\tilde{P}_{\sigma},x}$ are normal and Cohen-Macaulay by \cite[21.E]{mat1}, and they are also regular if $\mc{O}_{P_{\sigma},x}$ is. Hence (i) and (ii) follow.
\end{proof}

\begin{cor}\label{cor3.1}
The formal scheme $\mc{A}_{\Sigma}$ is normal and Cohen-Macaulay. It is also regular if $\Sigma$ is regular.
\end{cor}
\begin{proof}
Since normality, regularity and being Cohen-Macaulay are all local properties for \'etale topology, we are reduced to check for $\mr{Spf}A_{\sigma}$ for $\sigma\in\Sigma$, which follows from Lemma \ref{lem3.1}.
\end{proof}

\begin{prop}\label{prop3.1}
The algebraic space $A_{\Sigma}$ is normal and Cohen-Macaulay, and it is also regular if $\Sigma$ is regular.
\end{prop}
\begin{proof}
Since $(A_{\Sigma})_K$ is an abelian variety over $K$, it is a regular scheme. We are left to consider the points in $A_{\Sigma}\backslash (A_{\Sigma})_K$. For any $x\in A_{\Sigma}\backslash (A_{\Sigma})_K$, choose an open affine \'etale neighborhood $(U,u)$, with $U=\mr{Spec}B$ and $u$ lying over $x$, it suffices to investigate $u$ in $U$. 

Let $\hat{B}$ be the $\pi$-adic completion of $B$, and $\mc{U}=\mr{Spf}\hat{B}$. We have a canonical \'etale morphism $\mc{U}\rightarrow \mc{A}_{\Sigma}$ associated to $U$. Since $\mc{A}_{\Sigma}$ is normal and Cohen-Macaulay, so $\mc{U}$ is normal and Cohen-Macaulay, and the ring $\hat{B}$ is normal and Cohen-Macaulay. Consider the following commutative diagram
\begin{equation}
\xymatrix{
B\ar[rr]\ar[rd] & &\hat{B}  \\\
                &(1+(\pi))^{-1}B\ar[ru]
}.
\end{equation}
Since $\pi$ is contained in the Jacobson radical of $(1+(\pi))^{-1}B$, we have that $(1+(\pi))^{-1}B$ is normal and Cohen-Macaulay by \cite[7.8.3 (v)]{egaIV-2}. In particular, $B$ is normal and Cohen-Macaulay at $u$. It follows that the algebraic space $A_{\Sigma}$ is normal and Cohen-Macaulay. The regularity part can be proven by a similar argument.
\end{proof}

Now we define a canonical log structure $M$ (resp. $\mc{M}$) on $A_{\Sigma}$ (resp. $\mc{A}_{\Sigma}$) by letting
\begin{equation}\label{logstr}
M(U)=\{f\in\mc{O}_{A_{\Sigma}}(U)|f\in (\mc{O}_{A_{\Sigma}}(U)\otimes_RK)^{\times}\}
\end{equation}
\begin{equation}
(\mr{resp}.\quad \mc{M}(U)=\{f\in\mc{O}_{\mc{A}_{\Sigma}}(U)|f\in (\mc{O}_{\mc{A}_{\Sigma}}(U)\otimes_RK)^{\times}\})
\end{equation}
for any open $U$ of $(A_{\Sigma})_{\text{\'et}}$ (resp. $(\mc{A}_{\Sigma})_{\text{\'et}}$). This makes $A_{\Sigma}$ (resp. $\mc{A}_{\Sigma}$) into a log algebraic space (resp. log formal scheme) over the log scheme $S$ (resp. the log formal scheme $\mc{S}$). We have a canonical morphism 
\begin{equation}
\iota:\mc{A}_{\Sigma}\rightarrow A_{\Sigma}
\end{equation}
of ringed spaces. This further gives a morphism
\begin{equation}\label{sitesmorph}
(\mc{A}_{\Sigma})_{\text{\'et}}\rightarrow (A_{\Sigma})_{\text{\'et}}
\end{equation}
of small \'etale sites. Here for the definition of a morphism between sites, we refer to \cite[\href{http://stacks.math.columbia.edu/tag/00X0}{Tag 00X0}, Def. 7.14.1]{stacks-project}, for the proof of (\ref{sitesmorph}) being a morphism of sites, we make use of \cite[\href{http://stacks.math.columbia.edu/tag/00X0}{Tag 00X0}, Lem. 7.14.5]{stacks-project} (the category $(\mc{I}^u_V)^{\text{opp}}$ in \cite[\href{http://stacks.math.columbia.edu/tag/00X0}{Tag 00X0}, Lem. 7.14.5]{stacks-project} is obviously filtered in our case). Hence the pullback functor
$$\iota^{-1}:(\mc{A}_{\Sigma})_{\text{\'et}}^{\sim}\rightarrow (A_{\Sigma})_{\text{\'et}}^{\sim}$$
is exact.

Note that the log structures $M$ and $\mc{M}$ are defined in a similar way, and actually they are closely related along the morphism $\iota$. The structure morphism $M\rightarrow\mc{O}_{A_{\Sigma}}$ gives rise to 
$$\iota^{-1}M\rightarrow\iota^{-1}\mc{O}_{A_{\Sigma}}\rightarrow\mc{O}_{\mc{A}_{\Sigma}},$$
and let $\beta$ be the composition. It is easy to see that $\beta$ factors through $\mc{M}\rightarrow\mc{O}_{\mc{A}_{\Sigma}}$, in other words we have the following commutative diagram 
\begin{equation}
\xymatrix{
\iota^{-1}M\ar[d]\ar@{^{(}->}[r]\ar[rd]^{\beta}&\iota^{-1}\mc{O}_{A_{\Sigma}}\ar[d]   \\
\mc{M}\ar@{^{(}->}[r] &\mc{O}_{\mc{A}_{\Sigma}}
}.
\end{equation}
Let $\iota^*M$ be the inverse image of the log structure $M$, see \cite[1.4]{kat1} for the definition of the inverse image of log structure. By the definition of $\iota^*M$, we get a canonical homomorphism $\eta$
\begin{equation}
\xymatrix{
\beta^{-1}(\mc{O}_{\mc{A}_{\Sigma}}^{\times})\ar@{^{(}->}[r]\ar[d] &\iota^{-1}M\ar[d]\ar[rdd]  \\
\mc{O}_{\mc{A}_{\Sigma}}^{\times}\ar[r]\ar[rrd] &\iota^*M\ar@{.>}[rd]^{\eta}   \\
&&\mc{M}
}
\end{equation}
by the universal property of pushout. And $\eta$ is actually a homomorphism of log structures on $\mc{A}_{\Sigma}$.

\begin{prop}\label{prop3.2}
The canonical homomorphism $\eta$ is an isomorphism of log structures on $\mc{A}_{\Sigma}$.
\end{prop}
\begin{proof}
The homomorphism $\eta$ fits naturally into the following commutative diagram
\begin{equation}
\xymatrix{
1\ar[r] &\mc{O}_{\mc{A}_{\Sigma}}^{\times}\ar[r]\ar@{=}[d] &\iota^*M\ar[r]\ar[d]^\eta  &\overline{\iota^*M}\ar[r]\ar[d]^{\bar{\eta}} &1  \\
1\ar[r] &\mc{O}_{\mc{A}_{\Sigma}}^{\times}\ar[r]  &\mc{M}\ar[r]  &\overline{\mc{M}}\ar[r] &1
}
\end{equation}
with exact rows, where $\overline{\iota^*M}:=\iota^*M/\mc{O}^{\times}_{\mc{A}_{\Sigma}}$, $\overline{\mc{M}}:=\mc{M}/\mc{O}^{\times}_{\mc{A}_{\Sigma}}$. Note that here the sequence $1\rightarrow T'\xrightarrow{u} T\xrightarrow{v} T''\rightarrow 1$ being exact means that $T'$ is a subsheaf of groups of $T$ such that $T''$ is the associated quotient in the category of sheaves of monoids which is not an abelian category. To prove $\eta$ is an isomorphism, it is enough to show that $\bar{\eta}$ is an isomorphism. By \cite[3.3]{kat.fum1}, we have a canonical isomorphism $\overline{\iota^*M}\cong\iota^{-1}(\overline{M})$, where $\overline{M}:=M/\mc{O}^{\times}_{A_{\Sigma}}$. So we are left to prove that $\iota^{-1}(\overline{M})$ is isomorphic to $\overline{\mc{M}}$ under $\bar{\eta}$, and the injectivity and the surjectivity follow from part (1) and part (2) of the following Lemma \ref{lem3.2} respectively.
\end{proof}

\begin{lem}\label{lem3.2}
Let $B$ be a noetherian normal domain, $I$ an ideal in $B$, and $\hat{B}$ the $I$-adic completion of $B$. 
\begin{enumerate}[(1)]
\item Let $g_1,g_2\in B$ such that $V(g_1)\cup V(g_2)\subseteq V(I)$ and $g_1/g_2\in\hat{B}^{\times}$. Here $V(\cdot)$ denotes the vanishing set. Then we have $g_1/g_2\in B^{\times}$.
\item Let $f$ be an element in $\hat{B}$ which is not a zero-divisor, such that $V(f)\subseteq V(I\hat{B})=V(I)$. Then we can find a Zariski covering $\{U_j=\rm{Spec}B_j\}_{j\in J}$ of the scheme $U=\rm{Spec}B$, such that $f=g_ju_j$ on $\hat{B_j}$ for some $g_j\in B_j$ and $u_j\in \hat{B_j}^{\times}$. 
\end{enumerate} 
\end{lem}
\begin{proof}
First of all, we show part (1). Since $V(g_1)\cup V(g_2)\subseteq V(I)$ and $g_1/g_2\in\hat{B}^{\times}$, the principal Weil divisor for $g_1/g_2$ is zero. Hence $g_1/g_2\in B_{\mathfrak{p}}^{\times}$ for any prime ideal $\mathfrak{p}$ of height 1 of $B$. By \cite[Chap. 7, \S 17, Thm. 38]{mat1}, we get $g_1/g_2\in B^{\times}$.

Now we show part (2). The element $f$ defines an effective principal Cartier divisor, hence a closed subscheme of $\rm{Spec}\hat{B}$, see \cite[21.2.12]{egaIV-4}, which further gives a codim 1 cycle $D=\sum_i n_iD_i$. Now we regard $D$ as a codim 1 cycle on $\rm{Spec}B$, and by \cite[21.7.2]{egaIV-4} we get a closed subscheme $Y(D)$ of $U$. Let $I_D$ be the sheaf of ideals of $Y(D)$. Set-theoretically, $Y(D)$ is contained in $V(I)$, hence we have $I^r\subseteq I_D$ for some integer $r$ big enough. Then $I_D$ is an open ideal. It follows that $B/I_D\cong\hat{B}/I_D\hat{B}$, and $I_D\hat{B}$ gives rise to the cycle $D$ too. Again by \cite[21.7.2]{egaIV-4}, we get $(f)=I_D\hat{B}$.

Consider the following commutative diagram of $I$-adic rings
\begin{equation}
\xymatrix{
B\ar[rr]\ar[rd] & &\hat{B}  \\\
                &B'\ar[ru]
},
\end{equation}
where $B'$ denotes the ring $(1+I)^{-1}B$ which is a Zariski ring. Since $(f)=I_D\hat{B}=(I_DB')\hat{B}$, we have that $I_DB'$ is principal by \cite[24.E (i)]{mat1}. It is obvious that $I_D$ is principal on $U\backslash V(I)$, we only need to show that $I_D$ is principal at $B_{\mathfrak{p}}$ for all prime ideals $\mathfrak{p}\in V(I)$. But for $\mathfrak{p}\supseteq I$ we have $\mathfrak{p}\cap (1+I)=\emptyset$, hence $I_DB'$ being principal implies that $I_DB_{\mathfrak{p}}$ is principal.
\end{proof}

By Proposition \ref{prop3.2}, we could describe explicitly the pullback of the log structure $M$ on $A_{\Sigma}$ to $A_{\Sigma}\times_S S_n$.
\begin{cor}\label{cor3.2}
The pullback of the log structure $M$ on $A_{\Sigma}$ to $A_{\Sigma}\times_S S_n$ admits \'etale local charts $C(\sigma)^{\vee}\cap\mathbb{X}\longrightarrow R[C(\sigma)^{\vee}\cup\mathbb{X}]/(\pi^{n+1})$ with $\sigma$ varying in $\Sigma$.
\end{cor}

\section{Construction of logarithmic abelian varieties}
In this section, we are going to construct the log abelian variety over $S$ extending $A_K$. We follow the paper \cite{k-k-n2}, in particular section 5 there closely.

\subsection{}
In this subsection, we review some constructions from log geometry. Here we only work over $S$, but most constructions work over any fs log schemes. For any two sheaves of abelian groups $F$ and $G$ on $(\mr{fs}/S)$, we denote by $\mc{H}om(F,G)$ the hom-sheaf from $F$ to $G$.

Kato's log multiplicative group $\Gml$ is the sheaf of abelian groups on $(\mr{fs}/S)$ defined by
$$\Gml(U):=M_U^{\mr{gp}}(U)$$
for any $U\in (\mr{fs}/S)$, here $M_U$ denotes the log structure on $U$ and $M_U^{\mr{gp}}$ is the group envelope of $M_U$. It is easy to see that the multiplicative group $\Gm$ embeds into $\Gml$ canonically, so we have a canonical short exact sequence
$$0\rightarrow \Gm\rightarrow \Gml\rightarrow \Gml/\Gm\rightarrow 0.$$ 
We define the log torus $T_{\mr{log}}$ associated to $X$ to be $\mc{H}om(X,\Gml)$. Since
$$\mc{E}xt(X,\Gm)=0,$$
we get a canonical short exact sequence
$$0\rightarrow T\rightarrow T_{\mr{log}}\rightarrow\mc{H}om(X,\Gml/\Gm)\rightarrow 0$$
of sheaves of abelian groups on $(\mr{fs}/S)$.

Recall that we have a bilinear form $<,>:X\times Y\rightarrow K^{\times}$, see (\ref{eq1.1}). This gives rise to a bilinear form 
\begin{equation}\label{eq4.1}
<,>:X\times Y\rightarrow \Gml
\end{equation}
on $(\mr{fs}/S)$, hence a bilinear form 
\begin{equation}\label{eq4.2}
<,>:X\times Y\rightarrow \Gml/\Gm
\end{equation}
on $(\mr{fs}/S)$. Here by abuse of notation, we use $<,>$ to denote all the three parings, this shouldn't lead to any confusion since their meanings could be read out from the context. We define a subgroup sheaf $\mc{H}om(X,\Gml/\Gm)^{(Y)}$ of the sheaf $\mc{H}om(X,\Gml/\Gm)$ by
\[\begin{aligned}
&\mc{H}om(X,\Gml/\Gm)^{(Y)}(U) \\
:=&\left\{\varphi\in\mc{H}om(X,\Gml/\Gm)(U)\; \middle| \; \genfrac{}{}{0pt}{}{\forall u\in U,x\in X_{\bar{u}},\,\exists y,y'\in Y_{\bar{u}},}{\mr{s.t.} <x,y>|\,\varphi_{\bar{u}}(x)\,|<x,y'>} \right\}
\end{aligned}\]
for $U\in (\mr{fs}/S)$, here $\bar{u}$ denotes a geometric point above $u$. The pairing $<,>$ induces a homomorphism $Y\rightarrow \mc{H}om(X,\Gml/\Gm)$, which admits the following factorization 
$$\xymatrix{
Y\ar[rr]\ar[rd] &&\mc{H}om(X,\Gml/\Gm)  \\
&\mc{H}om(X,\Gml/\Gm)^{(Y)}\ar@{^(->}[ru]
}.
$$
Then we define a subgroup sheaf $\Tlog^{(Y)}$ of the sheaf $\Tlog$ as the inverse image of $\mc{H}om(X,\Gml/\Gm)^{(Y)}$ along the homomorphism $\Tlog\rightarrow \mc{H}om(X,\Gml/\Gm)$, whence a commutative diagram
$$\xymatrix{
0\ar[r] &T\ar[r]\ar@{=}[d] &\Tlog^{(Y)}\ar[r]\ar@{^(->}[d] &\mc{H}om(X,\Gml/\Gm)^{(Y)}\ar[r]\ar@{^(->}[d] &0  \\
0\ar[r] &T\ar[r] &\Tlog\ar[r] &\mc{H}om(X,\Gml/\Gm)\ar[r] &0
}$$
with exact rows.

Now for a given admissible polytope decomposition $\Sigma$ of $E$, we define a subsheaf $$\mc{H}om(X,\Gml/\Gm)^{(\Sigma)}$$ of the sheaf $$\mc{H}om(X,\Gml/\Gm)$$ by
\[\begin{aligned}
&\mc{H}om(X,\Gml/\Gm)^{(\Sigma)}(U) \\
:=&\left\{\varphi\in\mc{H}om(X,\Gml/\Gm)(U)\; \middle| \; \genfrac{}{}{0pt}{}{\forall u\in U,\,\exists\sigma\in\Sigma,\;\mr{s.t.}\,\forall (\mu,x)\in C(\sigma)^{\vee}\cap\mathbb{X},}{\mu\cdot\varphi(x)_{\bar{u}}\in (M_U/\mc{O}_U^{\times})_{\bar{u}}}  \right\}.
\end{aligned}\]
Here we stress again that being admissible always means being $Y$-admissible unless otherwise specified. Then we define a subsheaf $\Tlog^{(\Sigma)}$ of the sheaf $\Tlog$ as the inverse image of $$\mc{H}om(X,\Gml/\Gm)^{(\Sigma)}$$ along the homomorphism $$\Tlog\rightarrow \mc{H}om(X,\Gml/\Gm),$$
whence a commutative diagram
$$\xymatrix{
&T\ar[r]\ar@{=}[d] &\Tlog^{(\Sigma)}\ar[r]\ar@{^(->}[d] &\mc{H}om(X,\Gml/\Gm)^{(\Sigma)}\ar@{^(->}[d] \\
0\ar[r] &T\ar[r] &\Tlog\ar[r] &\mc{H}om(X,\Gml/\Gm)\ar[r] &0
}$$
with the second row exact.

\begin{rmk}
\begin{enumerate}
\item[(i)] The canonical inclusion $T\hookrightarrow\Tlog^{(Y)}$ restricts to an isomorphism over $K$, whilst the canonical inclusion $T\hookrightarrow\Tlog$ does not. The subgroup sheaf $\Tlog^{(Y)}$ cuts the very part related to the pairing $<,>$ out of $\Tlog$.
\item[(ii)] The sheaf $\mc{H}om(X,\Gml/\Gm)^{(\Sigma)}$ is actually a subsheaf of the sheaf \\
$\mc{H}om(X,\Gml/\Gm)^{(Y)}$, hence the sheaf $\Tlog^{(\Sigma)}$ is a subsheaf of the sheaf $\Tlog^{(Y)}$. But unlike $\mc{H}om(X,\Gml/\Gm)^{(Y)}$ (resp. $\Tlog^{(Y)}$) being a subgroup sheaf of the sheaf $\mc{H}om(X,\Gml/\Gm)$ (resp. $\Tlog$), the subsheaf
$\mc{H}om(X,\Gml/\Gm)^{(\Sigma)}$ (resp. $\Tlog^{(\Sigma)}$) is in general not a subgroup sheaf.
\end{enumerate}

\end{rmk}

There are many possible admissible polytope decompositions of $E$, whence many subsheaves of $\Tlog^{(Y)}$. We would like to know the relation between $\Tlog^{(\Sigma)}$ and $\Tlog^{(\Sigma')}$ for any two admissible polytope decompositions $\Sigma$ and $\Sigma'$. We would also like to know if it is possible to get $\Tlog^{(Y)}$ from the union of $\Tlog^{(\Sigma)}$ with $\Sigma$ varying in a certain family of admissible polytope decompositions of $E$. The answer is yes, if we also take $mY$-admissible polytope decompositions into account for $m\in\N$. We could even ask the representability of $\Tlog^{(\Sigma)}$ for certain $\Sigma$.  The followings answer these questions.

\begin{prop}\label{prop4.1.1}
For any two admissible polytope decompositions $\Sigma$ and $\Sigma'$ of $E$, we have $$\mc{H}om(X,\Gml/\Gm)^{(\Sigma)}\cap\mc{H}om(X,\Gml/\Gm)^{(\Sigma')}=\mc{H}om(X,\Gml/\Gm)^{(\Sigma\sqcap\Sigma')},$$
whence $\Tlog^{(\Sigma)}\cap \Tlog^{(\Sigma')}=\Tlog^{(\Sigma\sqcap\Sigma')}$.
\end{prop}
\begin{proof}
For $\sigma\in\Sigma$ and $\sigma'\in\Sigma'$, let $C(\sigma)$ and $C(\sigma')$ be as defined in part (b) of Fact \ref{fac2.1}. We have $C(\sigma)^{\vee}+C(\sigma')^{\vee}=C(\sigma\cap\sigma')^{\vee}$ by \cite[App. Thm. A.1 (2)]{oda1}. Then the results follow directly from the definitions of $\mc{H}om(X,\Gml/\Gm)^{(\Sigma)}$ and $\Tlog^{(\Sigma)}$.
\end{proof}

\begin{ex}\label{ex3}
Let the notation be as in Example \ref{ex1} and Example \ref{ex2}. For $m$ a positive integer, let $\Box^d_{m}$ be the $d$-cube with vertices $(a_1,\cdots,a_d)$, $a_i\in\{0,mn_i\}$, let $\Sigma_{\Box^d_{m}}$ be the $mY$-admissible polytope decomposition given by the $mY$-translates of the faces of $\Box^d_m$. Then we have $$\bigcup_{m\geq 1, a\in\frac{1}{2}Y}\Tlog^{(a+\Sigma_{\Box^d_{m}})}=\Tlog^{(Y)}$$ as sheaves, see \cite[3.5.4]{k-k-n1} for the proof. With the help of \cite[3.5.4]{k-k-n1}, one could also construct other examples.
\end{ex}

\begin{prop}\label{prop4.1.2}
Let $\Sigma$ be an $H$-admissible polytope decomposition of $E$ for a cofinite subgroup $H$ of $Y$, then the sheaf $\Tlog^{(\Sigma)}$ is represented by the fs log scheme over $S$ with underlying scheme $P_{\Sigma}$ and log structure coming from the monoids $(C(\sigma)^{\vee}\cap\mathbb{X})_{\sigma\in\Sigma}$.
\end{prop}
\begin{proof}
See \cite[3.5.3, 3.5.4]{k-k-n1}.
\end{proof}

\begin{cor}\label{cor4.1.2}
Let $H$ be a cofinite subgroup of $Y$, and $\Sigma$ an $H$-admissible polytope decomposition of $E$. We regard $A_{H,\Sigma}$ as a log algebraic space endowed with the log structure (\ref{logstr}).  Then the restriction of $A_{H,\Sigma}$ to $(\mr{fs}/S)'$ coincides with the restriction of $\Tlog^{(\Sigma)}/H$ to $(\mr{fs}/S)'$.
\end{cor}
\begin{proof}
This follows from Proposition \ref{prop4.1.2} and Corollary \ref{cor3.2}.
\end{proof}

\subsection{}
Let $H$ be a cofinite subgroup of $Y$, and $\Sigma$ an $H$-admissible polytope decomposition of $E$. We regard the algebraic space $A_{H,\Sigma}$ as a log algebraic space with respect to the canonical log structure defined in (\ref{logstr}).

Consider the sheaf of sets on $(\mr{fs}/S)$
\begin{equation}\label{deflav}
A=(\coprod_{(H,\Sigma)\in\mr{PolyDecom_Y}} A_{H,\Sigma})/\sim ,
\end{equation}
where $\sim$ is the equivalence relation in the category of sheaves on (fs/$S$) generated by the following equivalences:
\begin{enumerate}
 \item[(a)]  For any map $(H',\Sigma')\rightarrow (H,\Sigma)$ in $\mr{PolDecom_Y}$, i.e. $H'$ is a subgroup of $H$ and $\Sigma'$ is a subdivision of $\Sigma$, we have a canonical morphism
 $$A_{H',\Sigma'}\rightarrow A_{H,\Sigma}$$
 in (fs/$S$). Any element of $A_{H',\Sigma'}(U)$ for $U\in(\mr{fs}/S)$ is equivalent to its image in $A_{H,\Sigma}(U)$.
 \item[(b)] For any $(H,\Sigma)\in\mr{PolDecom_Y}$ and any $a\in Y$, we have a morphism $\mc{A}_{H,\Sigma}\rightarrow\mc{A}_{H,a+\Sigma}$ of formal schemes over $\mc{S}$ induced by the multiplication by the element $<\cdot,a>\in T_{\mr{log}}$, hence a morphism $A_{H,\Sigma}\rightarrow A_{H,a+\Sigma}$ of $S$-spaces. Any element of $A_{H,\Sigma}(U)$ for $U\in(\mr{fs}/S)$ is equivalent to its image in $A_{H,a+\Sigma}(U)$.
\end{enumerate}

The main results of this subsection are the following two theorems, which correspond to \cite[1.7, 4.7]{k-k-n2}.

\begin{thm}\label{thm4.1}
\begin{enumerate}[(a)]
 \item The pullback of $A$ to $(\mr{fs}/K)$ coincides with $A_K$. The restriction of $A$ to $(\mr{fs}/S)'$ coincides with $T_{\mr{log}}^{(Y)}/Y$.
 \item There exists a unique group law on $A$ whose pullback to $(\mr{fs}/K)$ coincides with the group law of $A_K$, and whose restriction to $(\mr{fs}/S)'$ coincides with the group law of $T_{\mr{log}}^{(Y)}/Y$.
 \item The canonical morphism $A_{Y,\Sigma}\rightarrow A$ fits into a Cartesian diagram of the form
$$\xymatrix{
A_{Y,\Sigma}\ar[r]^-{\beta_{Y,\Sigma}}\ar[d] & T_{\mr{log}}^{(\Sigma)}/(Y\cdot T)\ar@{^(->}[d]  \\
A\ar[r]^-{\beta} &T_{\mr{log}}^{(Y)}/(Y\cdot T)
}.
$$ 
In other words, the sheaf $A_{Y,\Sigma}$ coincides with the inverse image of $T_{\mr{log}}^{(\Sigma)}/(Y\cdot T)$ along $\beta$.
 \item With the group law specified in (b), $A$ fits into a short exact sequence $$0\rightarrow G\xrightarrow{\alpha} A\xrightarrow{\beta} T_{\mr{log}}^{(Y)}/(Y\cdot T)\rightarrow 0,$$ where $\alpha$ is the composition $G\hookrightarrow A_Y\hookrightarrow A$ with $A_Y$ as in Example \ref{ex1}, and $\beta$ is as in part (c).
\end{enumerate}
\end{thm}
\begin{proof}
The pullback of $A_{H,\Sigma}$ to $(\mr{fs}/K)$ is a Galois \'etale cover of $A_K$ with Galois group $Y/H$, hence the pullback of $A$ to $(\mr{fs}/K)$ is just $(A_{H,\Sigma})_K/(Y/H)=A_K$. By Example \ref{ex3} and Corollary \ref{cor4.1.2}, it is easy to see that the restriction of $A$ to $(\mr{fs}/S)'$ is just $\Tlog^{(Y)}$. This proves part (a).

The composition of the canonical morphisms
$$A\twoheadrightarrow i_*i^*A=i_*i^*(\Tlog^{(Y)}/Y)\twoheadrightarrow i_*i^*(\Tlog^{(Y)}/(Y\cdot T))=\Tlog^{(Y)}/(Y\cdot T)$$
gives a surjective homomorphism $\beta:A\twoheadrightarrow \Tlog^{(Y)}/(Y\cdot T)$. The latter equality can be seen from the expression $\Tlog^{(Y)}/(Y\cdot T)=\mc{H}om(X,\Gml/\Gm)^{(Y)}/\bar{Y}$, where $\bar{Y}$ denotes the image of the homomorphism $Y\rightarrow\mc{H}om(X,\Gml/\Gm)^{(Y)}$ associated to the paring $<,>$. Similarly we have a surjective morphism 
$$\beta_{H,\Gamma}:A_{H,\Gamma}\twoheadrightarrow \Tlog^{(\Gamma)}/(H\cdot T)$$
for any $H$-admissible polytope decomposition $\Gamma$ with $H$ a cofinite subgroup of $Y$. It is easy to see that $\beta$ and $\beta_{Y,\Sigma}$ fit into a canonical commutative diagram
\begin{equation*}
\xymatrix{A_{Y,\Sigma}\ar@{^(->}[d]\ar[r]\ar@{->>}[r] & \Tlog^{(\Sigma)}/Y\cdot T\ar@{^(->}[d]   \\
A\ar@{->>}[r] &\Tlog^{(Y)}/Y\cdot T 
}.
\end{equation*}
In order to finish the proof of (c), we need to show that: for any $U\in (\mr{fs}/S)$, and any $f\in A(U)$ such that $\beta(f)\in(\Tlog^{(\Sigma)}/(Y\cdot T))(U)$, \'etale locally on $U$ we can lift $f$ to a section of $A_{Y,\Sigma}$. Without loss of generality, we may assume that the underlying scheme of $U$ is $\mr{Spec}B$ for a noetherian ring $B$.

Suppose that $f$ is represented by a section in $A_{H,\Gamma}(B):=A_{H,\Gamma}(U)$ for some $(H,\Gamma)\in\mr{PolDecom}_Y$, and we still call it $f$. By using the diagram
\begin{equation*}
\xymatrix{&A_{H,\Gamma\sqcap\Sigma}\ar[ld]\ar[d]^{\text{\'etale quotient}}   \\
A_{H,\Gamma}   &A_{Y,\Gamma\sqcap\Sigma}\ar[d]  \\
&A_{Y,\Sigma}
},
\end{equation*}
it is enough to lift $f$ to $A_{H,\Gamma\sqcap\Sigma}$. Let $\hat{B}$ (resp. $B_K$) be $\varprojlim_{n\in\N} B/\pi^{n+1}B$ (resp. $B\otimes_R K$), then $f$ induces $\hat{f}=(f_n)_{n\in\N}\in A_{H,\Gamma}(\hat{B})=\varprojlim_n A_{H,\Gamma}(B/\pi^{n+1}B)=\varprojlim_n T^{(\Gamma)}_{\mr{log}}/H(B/\pi^{n+1}B)$ (resp. $f_K\in A_{H,\Gamma}(B_K)=A_K(B_K)$). Here the identification $A_{H,\Gamma}(\hat{B})=\varprojlim_n A_{H,\Gamma}(B/\pi^{n+1}B)$ is guaranteed by \cite[Thm. 1.1]{bha1}. We have $\Tlog^{(\Gamma)}\cap\Tlog^{(\Sigma)}=\Tlog^{(\Gamma\sqcap\Sigma)}$ by Proposition \ref{prop4.1.1} and $\beta(f)\in (\Tlog^{(\Sigma)}/(Y\cdot T))(B)$, hence $\beta(f)=\beta_{H,\Gamma}(f)\in (\Tlog^{(\Gamma\sqcap\Sigma)}/(H\cdot T))(B)$ and $\beta(f_n)\in (\Tlog^{(\Gamma\sqcap\Sigma)}/(H\cdot T))(B/\pi^{n+1}B)$. The diagram
$$\xymatrix{
T^{(\Gamma\sqcap\Sigma)}_{\mr{log}}/H\ar[r]\ar@{^(->}[d]  &T^{(\Gamma\sqcap\Sigma)}_{\mr{log}}/(H\cdot T)\ar@{^(->}[d] \\
T^{(\Gamma)}_{\mr{log}}/H\ar[r]  &T^{(\Gamma)}_{\mr{log}}/(H\cdot T)
}$$
is Cartesian, hence $f_n\in T^{(\Gamma\sqcap\Sigma)}_{\mr{log}}/H(B/\pi^{n+1}B)$. Note that $(f_n)_{n\in\N}$ is actually a compatible system, hence we get an element $\hat{g}=(f_n)_{n\in\N}\in \varprojlim_n T^{(\Gamma\sqcap\Sigma)}_{\mr{log}}/H(B/\pi^{n+1}B)=\varprojlim_n A_{H,\Gamma\sqcap\Sigma}(B/\pi^{n+1}B)=A_{H,\Gamma\sqcap\Sigma}(\hat{B})$. The classical fpqc covering $\hat{B}\times B_K$ of $B$ gives rise to the following commutative diagram
$$\xymatrix{
A_{H,\Gamma\sqcap\Sigma}(B)\ar[d]\ar@{^(->}[r] &A_{H,\Gamma\sqcap\Sigma}(\hat{B})\times A_K(B_K)\ar@<-.5ex>[r] \ar@<.5ex>[r]\ar[d] &A_K(\hat{B}\otimes_RK)\ar@{=}[d]\\
A_{H,\Gamma}(B)\ar@{^(->}[r] &A_{H,\Gamma}(\hat{B})\times A_K(B_K)\ar@<-.5ex>[r] \ar@<.5ex>[r] &A_K(\hat{B}\otimes_RK)
}$$
with exact rows. By diagram chasing, it is easy to see that $(\hat{g},f_K)$ descends to an element $g\in A_{H,\Gamma\sqcap\Sigma}(B)$ which  lifts $f$. This finishes the proof of part (c).

We postpone the proof of (b) and (d) to the end of this subsection.
\end{proof} 

\begin{thm}\label{thm4.2}
The sheaf $A$ with the group law specified in Theorem \ref{thm4.1} is a log abelian variety over $S$ extending the abelian variety $A_K$ over $K$.
\end{thm}
We also postpone the proof of Theorem \ref{thm4.2} to the end of this subsection.

\begin{defn}
In view of Theorem \ref{thm4.2}, we call the sheaf of abelian groups $A$ the \textbf{log abelian variety associated to the period lattice $Y$ in $T_K$}.
\end{defn}

\begin{rmk}
Any cofinite subgroup $H$ of $Y$ can be regarded as a period lattice in $T_K$ canonically, hence we can define the log abelian variety for $H$ in the same way as for $Y$ in (\ref{deflav}).
\end{rmk}

Before going to the rest parts of the proof of Theorem \ref{thm4.1} and the proof of Theorem \ref{thm4.2}, let us first give some lemmas needed for the proofs.

\begin{lem}\label{lem4.2.1}
For any $Y$-admissible polytope decomposition $\Sigma$ of $E$, the canonical map $A_{\Sigma}\rightarrow A$ is injective.
\end{lem}
\begin{proof}
This follows from part (c) of Theorem \ref{thm4.1}.
\end{proof}

\begin{cor}\label{cor4.2.1}
Let $H$ be a cofinite subgroup of $Y$, hence $H$ is also a period lattice in $T_K$. Let $A'$ be the log abelian variety associated to $H$. Then we have:
\begin{enumerate}
\item[(a)] The sheaf $A$ is an \'etale quotient of $A'$ under the canonical group action by $Y/H$.
\item[(b)] For any two $H$-admissible polytope decompositions $\Sigma$ and $\Gamma$, $A_{H,\Sigma}$ and $A_{H,\Gamma}$ are two subsheaves of $A'$ with $A_{H,\Sigma}\cap A_{H,\Gamma}=A_{H,\Sigma\sqcap\Gamma}$. In particular, for $H=Y$ we have $A_{\Sigma}\cap A_{\Gamma}=A_{\Sigma\sqcap\Gamma}$ inside $A$.
\end{enumerate}
\end{cor}
\begin{proof}
Part (a) follows from the definitions of $A$ and $A'$, see (\ref{deflav}). For (b), it is enough to consider the case $H=Y$. By Proposition \ref{prop4.1.1}, we have
$$\Tlog^{(\Sigma)}/(Y\cdot T)\cap \Tlog^{(\Gamma)}/(Y\cdot T)=\Tlog^{(\Sigma\sqcap\Gamma)}/(Y\cdot T)$$
inside $\Tlog^{(Y)}/(Y\cdot T)$. Then the result follows from the Cartesian diagram of part (c) of Theorem \ref{thm4.1}.
\end{proof}

\begin{lem}\label{lem4.2.2}
Let $U$ be an fs log scheme over $S$, and $f,g\in A(U)$. Let $\Box^d_{m}$ and $\Sigma_{\Box^d_{m}}$ be as in Example \ref{eq4.1}. Then \'etale locally on $U$, there exists an integer $m\geq 1$ and sections $\tilde{f},\tilde{g}$ of $A_{mY,\Sigma_{\Box^d_{m}}}$ such that $f$ (resp. $g$) comes from $\tilde{f}$ (resp. $\tilde{g}$), and $(\tilde{f},\tilde{g})$ belongs to $(A_{mY,\Sigma_{\Box^d_{m}}}\times_S A_{mY,\Sigma_{\Box^d_{m}}})_{\boxslash^{2d}}$. Here $(A_{mY,\Sigma_{\Box^d_{m}}}\times_S A_{mY,\Sigma_{\Box^d_{m}}})_{\boxslash^{2d}}$ is the analogue of $(A_Y\times_S A_Y)_{\boxslash^{2d}}=(A_{Y,\Sigma_{\Box^d}}\times_S A_{Y,\Sigma_{\Box^d}})_{\boxslash^{2d}}$ (see Example \ref{ex1}) for the period lattice $mY$.
\end{lem}
\begin{proof}
Suppose that $f$ (resp. $g$) comes from $A_{H,\Sigma}$ (resp. $A_{H',\Sigma'}$). For $\bar{\sigma}\in\Sigma/H$, we have an open algebraic subspace $A_{H,\Sigma,\bar{\sigma}}$ of $A_{H,\Sigma}$, see Definition \ref{defn2.3}. By Remark \ref{rmk2.1}, $(A_{H,\Sigma,\bar{\sigma}})_{\bar{\sigma}\in\Sigma/H}$ (resp. $(A_{H',\Sigma',\bar{\sigma}'})_{\bar{\sigma}'\in\Sigma'/H'}$) is an open covering of $A_{H,\Sigma}$ (resp. $A_{H',\Sigma'}$). We may assume that $f$ (resp. $g$) comes from $A_{H,\Sigma,\bar{\sigma}}(U)$ (resp. $A_{H',\Sigma',\bar{\sigma}'}(U)$) for some $\sigma\in\Sigma$ (resp. $\sigma'\in\Sigma'$). 

We can make $\sigma,\sigma'\subset-\underline{x}+\Box^d_{2m'}$ with $\underline{x}=(m'n_1,\cdots,m'n_d)$, and $m'Y\subset H\cap H'$, for some positive integer $m'$ big enough. According to the equivalences in (b) of Definition (\ref{deflav}), we can replace $\Sigma$ and $\Sigma'$ by their translates under $\underline{x}$, so that $\sigma,\sigma'\subset\Box^d_{2m'}$. Now let $m=4m'$, and we further replace $\Sigma'$ by its translate under $2\underline{x}=(2m'n_1,\cdots,2m'n_d)$, then we have $\sigma,\sigma'\subset\Box^d_m$, and $\sigma\times\sigma'$ goes into $\Box^d_m$ under the map $E\times E\rightarrow E,(a,b)\mapsto -a+b$. See the picture below.

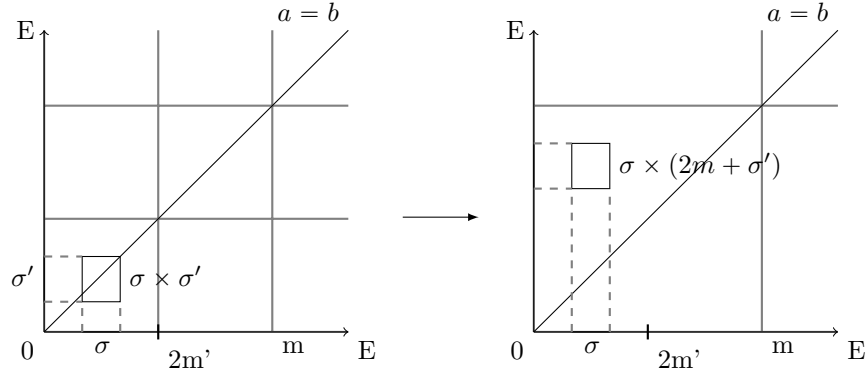
\tikzstyle{helplines}=[gray,thick,dashed]
\begin{tikzpicture}[domain=0:2]
\draw[thick,color=gray,step=1.5cm] (0,0) grid (4,4);
\draw [thick] (0,0) node[below left]{0} (3,0) node[below right]{m};
\draw [thick] (1.5,-.1) node[below right]{2m'} -- (1.5,.1);
\draw[->] (0,0) -- (4,0) node[below right]{E};
\draw[->] (0,0) -- (0,4) node[left] {E};
\draw (0,0) -- (4,4) node[above left]{$a=b$}; 
\draw (0.5,0)  -- (1,0) node[below left]{$\sigma$};
\draw (0,.4)  -- (0,1) node[below left]{$\sigma'$};
\draw (0.5,.4) rectangle (1,1) node[below right]{$\sigma\times\sigma'$};
\draw[style=helplines] (0.5,0) -- (0.5,.4);
\draw[style=helplines] (1,0) -- (1,.4);
\draw[style=helplines] (0,.4) -- (0.5,.4);
\draw[style=helplines] (0,1) -- (0.5,1);
\end{tikzpicture}
\begin{tikzpicture}
\draw (0,0)node{};
\draw[>=latex,->] (0,2) -- (1,2);
\end{tikzpicture}
\tikzstyle{helplines}=[gray,thick,dashed]
\begin{tikzpicture}[domain=0:2]
\draw[thick,color=gray,step=3cm] (0,0) grid (4,4);
\draw [thick] (0,0) node[below left]{0} (3,0) node[below right]{m};
\draw [thick] (1.5,-.1) node[below right]{2m'} -- (1.5,.1);
\draw[->] (0,0) -- (4,0) node[below right] {E};
\draw[->] (0,0) -- (0,4) node[left] {E};
\draw (0,0) -- (4,4) node[above left]{$a=b$}; 
\draw (0.5,0)  -- (1,0) node[below left]{$\sigma$};
\draw (0,1.9)  -- (0,2.5) node[below left]{$2\underline{x}+\sigma'$};
\draw (0.5,1.9) rectangle (1,2.5);
\draw (0.5,2.5) node[above right]{$\sigma\times(2\underline{x}+\sigma')$};
\draw[style=helplines] (0.5,0) -- (0.5,1.9);
\draw[style=helplines] (1,0) -- (1,1.9);
\draw[style=helplines] (0,1.9) -- (0.5,1.9);
\draw[style=helplines] (0,2.5) -- (0.5,2.5);
\end{tikzpicture}   \\
So we get the following diagram
\begin{equation*}
\xymatrix{
A_{mY,\Sigma,\bar{\sigma}}\ar[r]^-{\alpha}\ar[d]_{\text{\'etale}}&A_{mY,\Sigma_{\Box^d_{m}},\bar{\Box}^d_m} &A_{mY,\Sigma',\bar{\sigma}'}\ar[l]_-{\alpha'}\ar[d]^{\text{\'etale}} \\
A_{H,\Sigma,\bar{\sigma}} &&A_{H',\Sigma',\bar{\sigma}'}
}
\end{equation*}
with the following canonical factorization
\begin{equation*}
\xymatrix{
A_{mY,\Sigma,\bar{\sigma}}\times_S A_{mY,\Sigma',\bar{\sigma}'}\ar[d]^-{\alpha\times\alpha'}\ar[rd] \\
A_{mY,\Sigma_{\Box^d_{m}},\bar{\Box}^d_m}\times_S A_{mY,\Sigma_{\Box^d_{m}},\bar{\Box}^d_m}\ar[d] &(A_{mY,\Sigma_{\Box^d_{m}}}\times_S A_{mY,\Sigma_{\Box^d_{m}}})_{\boxslash^{2d}}\ar[ld]  \\
A_{mY,\Sigma_{\Box^d_{m}}}\times_S A_{mY,\Sigma_{\Box^d_{m}}}  \\
}.
\end{equation*}
It follows that, \'etale locally, we can make $f$ and $g$ come from $\tilde{f}$ and $\tilde{g}$ in $A_{mY,\Sigma_{\Box^d_{m}}}$ respectively such that $(\tilde{f},\tilde{g})$ belongs to $(A_{mY,\Sigma_{\Box^d_{m}}}\times_S A_{mY,\Sigma_{\Box^d_{m}}})_{\boxslash^{2d}}$.
\end{proof}

\begin{lem}\label{lem4.2.3}
Let $\mr{Spec} B$ be an fs log scheme over $S$ with $B$ a noetherian ring. Then the canonical map $A(B)\rightarrow A(B\otimes_R K)\times\varprojlim_n A(B/\pi^{n+1}B)$ is injective.
\end{lem}
\begin{proof}
If $B$ is a $K$-algebra, there is nothing needed to prove. We assume that $\pi$ is not invertible in $B$. We may further assume that $\mr{Spec}B$ is connected. Let $f,g\in A(B)$, we assume that the image of $f$ and $g$ in $A(B\otimes_R K)\times\varprojlim_n A(B/\pi^{n+1}B)$ coincide. We want to prove $f=g$.

By Lemma \ref{lem4.2.2}, \'etale locally on $\mr{Spec} B$, there exists $m\geq 1$ such that $f,g$ come from $\tilde{f},\tilde{g}\in A_{mY,\Sigma_{\Box^d_{m}}}$ and such that $(\tilde{f},\tilde{g})\in (A_{mY,\Sigma_{\Box^d_{m}}}\times_S A_{mY,\Sigma_{\Box^d_{m}}})_{\boxslash^{2d}}$. Since the images of $f,g$ in $T_{\mr{log}}^{(Y)}/(T\cdot Y)$ coincide, there exists $y\in Y$ such that the images of $\tilde{h}:=<\cdot,y>\tilde{f}$ and $\tilde{g}$ in $T_{\mr{log}}^{(Y)}/(T \cdot(mY))$ coincide. Hence the image of $\tilde{h}$ in $T_{\mr{log}}^{(Y)}/(T \cdot(mY))$ also lies in $T_{\mr{log}}^{(\Sigma_{\Box^d_{m}})}/(T \cdot(mY))$. It follows that $\tilde{h}\in A_{mY,\Sigma_{\Box^d_{m}}}$ by part (c) of Theorem \ref{thm4.1} applied to the abelian variety associated to the period lattice $mY$ in $T_K$. We also have that the image of $(\tilde{h},\tilde{g})$ in
$$(T\times_ST)_{\mr{log}}^{(mY\times mY)}/(T\cdot (mY\times mY))$$
actually lies in 
$$(T\times_ST)_{\mr{log}}^{(\Sigma_{\boxslash^{2d},mY})}/(T\cdot (mY\times mY)).$$
Here $\Sigma_{\boxslash^{2d},mY}$ is an $mY\times mY$-admissible polytope decomposition of $E\times E$ which is defined to be the analogue of $\Sigma_{\boxslash^{2d}}$ (see Example \ref{ex1}) for the period lattice $mY\times mY$ in $(T\times_S T)_K$. By part (c) of Theorem \ref{thm4.1} applied to the abelian variety associated to the period lattice $mY\times mY$ in $(T\times_S T)_K$, we have that
$$(\tilde{h},\tilde{g})\in (A_{mY,\Sigma_{\Box^d_{m}}}\times_S A_{mY,\Sigma_{\Box^d_{m}}})_{\boxslash^{2d}}.$$
Let $G'$ be the semi-abelian scheme over $S$ corresponding to the period lattice $mY$ in $T$. Since the image of $a:=m_{-}(\tilde{h},\tilde{g})$ in $T_{\mr{log}}^{(\Sigma_{\Box^d_{m}})}/(T \cdot(mY))$ is the identity of $T_{\mr{log}}^{(Y)}/(T \cdot(mY))$, the section $a$ of $A_{mY,\Sigma_{\Box^d_{m}}}$ actually lies in the open subspace $G'$ of $A_{mY,\Sigma_{\Box^d_{m}}}$, i.e. $a\in G'(B)$. Let $G$ be the semi-abelian scheme over $S$ associated to $A_K$. The image of $a$ in $G(B\otimes_RK)$ and the image of $a$ in $\varprojlim_nG(B/\pi^{n+1}B)=G(\varprojlim_nB/\pi^{n+1}B)$ vanish. Since $B\rightarrow B\otimes_RK\times \varprojlim_nB/\pi^{n+1}B$ is faithfully flat, the canonical map $G(B)\rightarrow G(B\otimes_RK)\times \varprojlim_nG(B/\pi^{n+1}B)$ is injective. Hence the image of $a$ in $G$ is the unit element, therefore $a$ lies in the quasi-finite separated group scheme $F:=\mr{Ker}(G'\rightarrow G)$. Since the special fibre of $F$ is trivial, the underlying scheme of $F$ is the disjoint union of the unit section and another open subscheme defined over $K$. Since $\pi$ is not invertible in $B$ and $\mr{Spec}B$ is connected, $a$ is actually the identity element of $F(B)$. 

Now we consider the following commutative diagram
\begin{equation*}
\xymatrix{A_{mY,\Sigma_{\Box^d_{m}}}(B)\times A_{mY,\Sigma_{\Box^d_{m}}}(B)\ar@{^(->}[r]^-{\alpha_2\times \alpha_2} &Q_{mY,\Sigma_{\Box^d_{m}}}\times Q_{mY,\Sigma_{\Box^d_{m}}}   \\
(A_{mY,\Sigma_{\Box^d_{m}}}\times_S A_{mY,\Sigma_{\Box^d_{m}}})_{\boxslash^{2d}}(B)\ar@{^(->}[r]^-{\alpha_1}\ar[d]^{m_-}\ar[u] &(Q_{mY,\Sigma_{\Box^d_{m}}}\times Q_{mY,\Sigma_{\Box^d_{m}}})_{\boxslash^{2d}}
\ar@{^(->}[u]\ar[d]^{\tilde{m}_-} \\
A_{mY,\Sigma_{\Box^d_{m}}}(B)\ar@{^(->}[r]^-{\alpha_2} &Q_{mY,\Sigma_{\Box^d_{m}}}
},
\end{equation*}
where $Q_{mY,\Sigma_{\Box^d_{m}}}$ denotes the set 
$$A_{mY,\Sigma_{\Box^d_{m}}}(B\otimes_R K)\times\varprojlim_n A_{mY,\Sigma_{\Box^d_{m}}}(B/\pi^{n+1}B),$$ 
and $(Q_{mY,\Sigma_{\Box^d_{m}}}\times Q_{mY,\Sigma_{\Box^d_{m}}})_{\boxslash^{2d}}$ denotes the set 
$$(A_{mY,\Sigma_{\Box^d_{m}}}\times_S A_{mY,\Sigma_{\Box^d_{m}}})_{\boxslash^{2d}}(B\otimes_R K)\times \varprojlim_n (A_{mY,\Sigma_{\Box^d_{m}}}\times_S A_{mY,\Sigma_{\Box^d_{m}}})_{\boxslash^{2d}}(B/\pi^{n+1}B).$$
The map $\tilde{m}_-$ is the restriction of the map $\mathscr{G}\times\mathscr{G}\rightarrow\mathscr{G},(x,y)\mapsto x^{-1}y$, where $\mathscr{G}$ denotes the group $$A_{mY,\Sigma_{\Box^d_{m}}}(B\otimes_R K)\times\varprojlim_n T^{(mY)}_{\mr{log}}/mY(B/\pi^{n+1}B).$$
Then $a=m_-(\tilde{h},\tilde{g})$ being the identity of $G'(B)$ implies that
$$\tilde{m}_-(\alpha_1(\tilde{h},\tilde{g}))= \alpha_2(\tilde{h})^{-1}\alpha_2(\tilde{g})=1$$
in the group $\mathscr{G}$. Hence we get $<\cdot,y>\tilde{f}=\tilde{h}=\tilde{g}$ from the injectivity of $\alpha_2$. It follows then $f=g$.
\end{proof}

\begin{lem}\label{lem4.2.4}
In the situation of Lemma \ref{lem4.2.3}, $A(B)$ is a subgroup of the abelian group $A(B\otimes_R K)\times\varprojlim_n A(B/\pi^{n+1}B)$. 
\end{lem}
\begin{proof}
Let $f,g\in A(B)$. By Lemma \ref{lem4.2.2}, we can assume that $f,g$ come from $\tilde{f},\tilde{g}\in A_{mY,\Sigma_{\Box^d_{m}}}$ for some $m\geq 1$, such that $(\tilde{f},\tilde{g})$ lies in $(A_{mY,\Sigma_{\Box^d_{m}}}\times_S A_{mY,\Sigma_{\Box^d_{m}}})_{\boxslash^{2d}}$. Write the group $A(B\otimes_R K)\times\varprojlim_n A(B/\pi^{n+1}B)$ as $W$, we have the following commutative diagram
\begin{equation*}
\xymatrix{
(\tilde{f},\tilde{g})\in\quad (A_{mY,\Sigma_{\Box^d_{m}}}\times_S A_{mY,\Sigma_{\Box^d_{m}}})_{\boxslash^{2d}}(B)\ar[d]\ar[rd]^{m_{-}}  \\
(\tilde{f},\tilde{g})\in\quad A_{mY,\Sigma_{\Box^d_{m}}}(B)\times A_{mY,\Sigma_{\Box^d_{m}}}(B)\ar[d]  &A_{mY,\Sigma_{\Box^d_{m}}}(B)\ar[d] \quad \ni m_{-}(\tilde{f},\tilde{g})\\
(f,g)\in\quad A(B)\times A(B)\ar@{^(->}[d]  &  A(B)\ar@{^(->}[d] \\
W\times W\ar[r]_{(a,b)\mapsto a^{-1}\cdot b}  &W
}
\end{equation*}
Hence as a subset of $W$, $A(B)$ is closed under the group operation of $W$, hence a subgroup of $W$.
\end{proof}

\begin{proof}[Proof of Theorem \ref{thm4.1}, continued:]
Now we come to the proof of part (b). \\
Combining Lemma \ref{lem4.2.3} and Lemma \ref{lem4.2.4}, by a limit argument which reduces the problem to noetherian rings, we get a unique group structure on $A$ with required properties.

For part (d), the injectivity of $\alpha$ and the surjectivity of $\beta$ are clear, so we are left to show the exactness in the middle. Further investigation of the definitions of $\alpha$ and $\beta$, tells us $\beta\circ\alpha=0$. Now let $f\in A(\mr{Spec}B)$ with $\beta(f)=0$, where $\mr{Spec}B$ is a noetherian fs log scheme over $S$ . By (c), we know $f\in A_Y(\mr{Spec}B)$ with $f_0=f\times_S S_0=0$, hence the image of $f$ lies in the open subspace $G$ of $A_Y$. Here $A_Y$ is as defined in Example \ref{ex1}. This finishes the proof of part (d).
\end{proof}

\begin{proof}[Proof of Theorem \ref{thm4.2}:]
The condition \cite[4.1.1]{k-k-n2} is satisfied by part (a) of Theorem \ref{thm4.1}. By part (d) of Theorem \ref{thm4.1}, the condition \cite[4.1.2]{k-k-n2} is also satisfied. We are left to show the separateness condition \cite[4.1.3]{k-k-n2}. For $U\in (\mathrm{fs}/S)$ and two morphisms $f,g:U\rightarrow A$, we need to show that the equalizer $E(f,g)$ is finite over $U$. Let $h:=f\cdot g^{-1}$, then we have $E(f,g)=E(h,1)$. Suppose that $h$ comes from a section $h'$ of $A_{mY,\Sigma_{\Box_m^d}}$ \'etale locally for some $m$. Let $A_K'$ be the abelian variety with period lattice $mY$ for a positive integer $m$, let $A'$ be the log abelian variety corresponding to $A_K'$. By construction (\ref{deflav}), we have a short exact sequence $$0\rightarrow Y/mY\rightarrow A'\rightarrow A\rightarrow 0$$ of abelian sheaves over $(\mathrm{fs}/S)$. By the above short exact sequence, the equalizer $E(h,1)$ is locally the disjoint union of $E(h',a\cdot 1)$ with $a$ varying in $Y/mY$. Note that $a\cdot 1$ is a section of $A_{mY,\Sigma_{\Box_m^d}}$ for $a\in Y/mY$. The algebraic space $A_{mY,\Sigma_{\Box_m^d}}$ is separated over $S$, hence $E(h',a\cdot 1)$ is finite over $U$, so is $E(h,1)$. This finishes the proof.
\end{proof}

\subsection{}
In last subsection, we have associated a canonical log abelian variety $A$ over $S$ to any split totally degenerate abelian variety $A_K$ over $K$. Let $\mr{STDAV_K}$ (resp. $\mr{LAV}_S$) denote the category of split totally degenerate abelian varieties over $K$ (resp. log abelian varieties over $S$). We regard this association as a map 
$$\mr{Deg}:\mr{STDAV}_K\rightarrow \mr{LAV}_S.$$ 
In this subsection, we show that $\mr{Deg}$ is actually a functor. In particular, we will associate to any homomorphism $f:A_K\rightarrow A_K'$ in $\mr{STDAV}_K$ a homomorphism from $A:=\mr{Deg}(A_K)$ to $A':=\mr{Deg}(A_K')$.

Let $T$ (resp. $T'$) and $Y$ (resp. $Y'$) be the Raynaud extension and the period lattice of $A_K$ (resp. $A_K'$) respectively, and let $Y\xrightarrow{u_K} T_K$ (resp. $Y'\xrightarrow{u'_K} T_K'$) be the rigid analytic uniformisation of $A_K$ (resp. $A_K'$). Then we get a homomorphism 
$$\begin{CD}
Y @>f_{-1}>> Y'\\
@VVu_KV @VVu_K'V\\
T_K @>f_0>> T_K'
\end{CD} $$
of 1-motives over $K$ in a functorial way. This further gives the following commutative diagram
$$\begin{CD}
Y @>f_{-1}>> Y'\\
@VVu_KV @VVu_K'V\\
E=\mr{Hom}(X,\Q) @>f_a>> E'=\mr{Hom}(X',\Q)
\end{CD} $$
after taking valuation maps and tensoring with $\Q$. 

Now in order to have a map from $A$ to $A'$, we need to find $A_{H,\Sigma}$ somewhere to go for each pair $(H,\Sigma)\in\mr{PolDecom}_Y$. This comes down to find a pair $(H',\Sigma')\in\mr{PolDecom}_{Y'}$ such that $f_{-1}(H)\subset H'$ and for any $\sigma\in\Sigma$ there exists $\sigma'\in\Sigma'$ with $f_a(\sigma)\subset\sigma'$. We are going to construct such a pair explicitly.

Let $\overline{f_a(Y)}:=f_a(E)\cap Y'$, we have that $Y'=\overline{f_{-1}(Y)}\oplus\tilde{Y}'$ for some torsion-free subgroup $\tilde{Y}'\leq Y'$ (we fix $\tilde{Y}'$ from now on), and $E'=f_{a}(E)\oplus\tilde{E}'$ with $\tilde{E}':=\tilde{Y}'\otimes\Q$. After choosing a $\Z$-basis of $\tilde{Y}'$, we make a $\tilde{Y}'$-admissible polytope decomposition $\Lambda$ of $\tilde{E}'$ by the cubes with respect to this basis. Note that $f_a(\Sigma)$ is an $f_a(H)$-admissible polytope decomposition of $f_a(E)$. We make a polytope decomposition $\Sigma'$ by the product of $f_a(\Sigma)$ and $\Lambda$. Let $H':=f_a(H)\oplus\tilde{Y}'$, it is clear that $\Sigma'$ is $H'$-admissible. So we get a pair
\begin{equation}\label{constr4.3.1}
(H',\Sigma')\in\mr{PolDecom}_{Y'}
\end{equation} 
with required properties. It follows that we get a morphism $A_{H,\Sigma}\rightarrow A_{H',\Sigma'}'$ of proper algebraic $S$-spaces. We map $A_{H,\Sigma}$ into $A'$ by the composition $A_{H,\Sigma}\rightarrow A_{H',\Sigma'}'\rightarrow A'$.

The map $A_{H,\Sigma}\rightarrow A'$ is independent of the choice of $(H',\Sigma')$. If we are given another such pair $(H'_1,\Sigma'_1)$, then the pair $(H'\cap H_1',\Sigma'\sqcap\Sigma'_1)$ is a third such pair. Hence we get the following commutative diagram
\begin{equation*}
\xymatrix{
                 &A_{H,\Sigma}\ar[ld]\ar[rd]\ar[d] \\
A_{H',\Sigma'}'  &A_{H'\cap H'_1,\Sigma'\sqcap\Sigma'_1}'\ar[r]\ar[l] &A_{H'_1,\Sigma'_1}'
}
\end{equation*}
which guarantees the well-definedness of the map $A_{H,\Sigma}\rightarrow A'$. We denote it by $\mr{Deg}(f)_{H,\Sigma}$.

Now we want show that the collection $(\mr{Deg}(f)_{H,\Sigma})_{(H,\Sigma)}$ glues into a map from $A$ to $A'$. In other words, the collection $(\mr{Deg}(f)_{H,\Sigma})_{(H,\Sigma)}$ is compatible with the equivalence relation in (\ref{deflav}). Given a map $(H_1,\Sigma_1)\rightarrow (H_2,\Sigma_2)$ in $\mr{PolDecom}_Y$, hence $H_1$ is a subgroup of $H_2$ and $\Sigma_1$ is a subdivision of $\Sigma_2$. The construction (\ref{constr4.3.1}) is compatible with $(H_1,\Sigma_1)\rightarrow (H_2,\Sigma_2)$, hence we get a commutative diagram
$$\xymatrix{
A_{H_1,\Sigma_1} \ar[r]\ar[d] &A_{H_2,\Sigma_2}\ar[d] \\
A_{H_1',\Sigma_1'}'\ar[r] &A_{H_2',\Sigma_2'}'
}.$$
This shows the compatibility for the first kind of equivalences in (\ref{deflav}). Similar argument works for the second kind of equivalences in (\ref{deflav}) too. Hence the collection $(\mr{Deg}(f)_{H,\Sigma})_{(H,\Sigma)}$ glues into a map $\mr{Deg}(f):A\rightarrow A'$. Since the group law of $A$ is defined with respect to Lemma \ref{lem4.2.4}, the map $\mr{Deg}(f)$ is actually a homomorphism. The construction of $\mr{Deg}$ is clearly functorial, whence the following theorem.

\begin{thm}\label{thm4.3}
The association of a log abelian variety to any split totally degenerate abelian variety over $K$ gives rise to a functor $\mr{Deg}:\mr{STDAV}_K\rightarrow \mr{LAV}_S$. Moreover, the functor $A\mapsto A_K:=A\times_S \mr{Spec}K$ is left inverse to $\mr{Deg}$ and $\mr{Deg}$ is fully faithful.
\end{thm}
\begin{proof}
The functor $A\mapsto A_K$ is clearly left inverse to Deg by the construction of Deg. We only need to check the fullness. 

Given any two log abelian varieties $A_1$ and $A_2$ over $S$, let
$$0\rightarrow G_1\rightarrow A_1\rightarrow \mc{H}om_S(X_1,\Gml/\Gm)^{(Y_1)}/\bar{Y_1}\rightarrow 0$$
and 
$$0\rightarrow G_2\rightarrow A_2\rightarrow \mc{H}om_S(X_2,\Gml/\Gm)^{(Y_2)}/\bar{Y_2}\rightarrow 0$$
be the corresponding short exact sequences associated to the log abelian varieties $A_1$ and $A_2$ (see \cite[Def. 4.1, 4.1.2]{k-k-n2}) respectively. Let $f:A_1\rightarrow A_2$ be a homomorphism of log abelian varieties, then $f$ induces a homomorphism
$$\tilde{f}:G_1\rightarrow\mc{H}om_S(X_2,\Gml/\Gm)^{(Y_2)}/\bar{Y_2}.$$
Since $\tilde{f}$ is zero by \cite[9.2]{k-k-n2}, $f$ induces homomorphisms $g:G_1\rightarrow G_2$ and 
$$h:\mc{H}om_S(X_1,\Gml/\Gm)^{(Y_1)}/\bar{Y_1}\rightarrow \mc{H}om_S(X_2,\Gml/\Gm)^{(Y_2)}/\bar{Y_2}.$$
The pair $(g,h)$ determines $f$ by Lemma \ref{lem4.2.5} below. It is also clear that $g_K=f_K$ determines $g$. We are going to show that $g$ determines $h$ in the case that both $A_1$ and $A_2$ lie in the image of Deg. Then $f_K$ determines $f$, hence $f=\mr{Deg}(f_K)$ and the fullness follows. 

Now suppose that $A_1$ and $A_2$ lie in the image of the functor Deg, and they are constructed out of some bilinear pairings $<,>_i:X_i\times Y_i\rightarrow \Gml$, see (\ref{eq4.1}). The pairing $<,>_i$ gives rise to $ \mc{H}om_S(X_i,\Gml/\Gm)^{(Y_i)}/\bar{Y_i}$, and it corresponds to a log 1-motive $[Y_i\rightarrow (T_i)_{\mr{log}}]$. The log 1-motive $[Y_i\rightarrow (T_i)_{\mr{log}}]$ can be recovered from $G_i$ through the equivalence of categories in \cite[Chap. III, Prop. 6.4]{f-c}. The map $g$ is nothing but a map between two elements of $\mr{DEG}_{\mr{ample}}$ forgetting the invertible sheaves, while the map $h$ can be obtained from the map of log 1-motives corresponding to $g$. Hence the map $h$ is determined by the map $g$. This finishes the proof.
\end{proof}

\begin{lem}\label{lem4.2.5}
Let $G_2$ and $\mc{H}om_S(X_1,\Gml/\Gm)^{(Y_1)}/\bar{Y_1}$ be as in the proof of Theorem \ref{thm4.3}. Then $\mr{Hom}_S(\mc{H}om_S(X_1,\Gml/\Gm)^{(Y_1)}/\bar{Y_1},G_2)=0$.
\end{lem}
\begin{proof}
Let $\mc{Q}_1$ denote $\mc{H}om_S(X_1,\Gml/\Gm)^{(Y_1)}=T^{(Y_1)}_{1\mr{log}}/T_1$, where $T_1$ denotes $\mc{H}om_S(X_1,\Gm)$. If suffices to show that $\mr{Hom}_S(\mc{Q}_1,G_2)=0$. The proof of $\mr{Hom}_S(\mc{Q}_1,G_2)=0$ is motivated by the proof of \cite[Lem. 6.1.5]{k-k-n2}.

Let $E_1=\mr{Hom}(X_1,\Q)$, $\mathbb{X}_1=\pi^{\Z}\oplus X$, and $\mathbb{E}_1=\mr{Hom}(\mathbb{X}_1,\Q)\cong\Q\oplus E_1$. For a polytope $\sigma$ inside $E_1$, we denote by $C(\sigma)$ the cone in $\mathbb{E}_1$ above $\sigma$ with $E_1$ identified with the hyperplane $\{1\}\times E_1$ in $\mathbb{E}_1$. We denote by $P_{\sigma}$ the log scheme $\mr{Spec}R[C(\sigma)^{\vee}\cap\mathbb{X}_1]$ endowed with the log structure associated to $C(\sigma)^{\vee}\cap\mathbb{X}_1\rightarrow R[C(\sigma)^{\vee}\cap\mathbb{X}_1]$.

Now let $U\in(\mr{fs}/S)$. By \cite[\S 7.7]{k-k-n2}, any section $\alpha$ of $T^{(Y_1)}_{1\mr{log}}/T_1$ on $U$ can be obtained, \'etale locally on $U$, as $a^*(u)\mathop{\mr{mod}}T_1$ for some $a\in P_{\sigma}(U)$ \footnote{In terms of the notation of \cite[\S 7.7]{k-k-n2}, $P_{\sigma}$ is the log scheme representing the subsheaf $V(C(\sigma)^{\vee}\cap\mathbb{X}_1)$ of $T^{(Y_1)}_{1\mr{log}}$.} for some polytope $\sigma$ inside $E_1$, where $u\in T^{(Y_1)}_{1\mr{log}}(P_{\sigma})$ is the universal section
$$x\in X_1\mapsto x\in M_{P_{\sigma}}^{\mr{gp}}.$$ 
For any homomorphism $\varphi:T^{(Y_1)}_{1\mr{log}}/T_1\rightarrow G_2$, we have 
$$\varphi(\alpha)=\varphi(a^*(u)\mathop{\mr{mod}}T_1)=a^*(\varphi(u\mathop{\mr{mod}}T_1)).$$
Hence we are reduced to show that $\varphi(u\mathop{\mr{mod}}T_1)=0$. Since $u\mathop{\mr{mod}}T_1=0$ vanishes on $P_{\sigma}\times_S\mr{Spec}K$, the section $\varphi(u\mathop{\mr{mod}}T_1)\in G_2(P_{\sigma})$ vanishes in $G_2(P_{\sigma}\times_S\mr{Spec}K)$. Since $G_2$ is separated over $S$ and $P_{\sigma}$ is reduced, the map $G_2(P_{\sigma})\rightarrow G_2(P_{\sigma}\times_S\mr{Spec}K)$ is injective. It follows that $\varphi(u\mathop{\mr{mod}}T_1)=0$.
\end{proof}

\begin{rmk}
In future we hope to generalise the results in Theorem \ref{thm4.2} and Theorem \ref{thm4.3} to all semi-stable abelian varieties over $K$. In the case that the Rayanud extension of $A_K$ has split torus part, the result follows probably from the split totally degenerate case with the help of contracted product. As stated in \cite[13.4]{k-k-n4}, the category of semi-stable abelian varieties over $K$ is canonically equivalent to the category of log abelian varieties over $S$. Our functor Deg should be the restriction of this equivalence to $\mr{STDAV}_K$. Such an equivalence justifies the motto ``log abelian varieties are canonical degenerations of abelian varieties''.
\end{rmk}

\bibliographystyle{alpha}
\bibliography{bib}

\end{document}